\documentclass[11pt,a4paper,reqno]{amsart}
\usepackage[T1]{fontenc}
\usepackage{lmodern,amsmath,amssymb,amsthm,mathtools}
\usepackage[margin=2.6cm]{geometry}
\usepackage{enumitem,booktabs,listings}
\usepackage[colorlinks=true,linkcolor=blue,citecolor=blue,urlcolor=blue]{hyperref}
\newtheorem{theorem}{Theorem}[section]
\newtheorem{lemma}[theorem]{Lemma}
\newtheorem{proposition}[theorem]{Proposition}
\newtheorem{corollary}[theorem]{Corollary}
\theoremstyle{definition}

\newtheorem{example}[theorem]{Example}
\newtheorem{remark}[theorem]{Remark}
\DeclareMathOperator{\Aut}{Aut}
\newcommand{\PhiG}{\Phi(G)}
\DeclareMathOperator{\ord}{ord}
\DeclareMathOperator{\id}{id}
\DeclareMathOperator{\cl}{cl}
\newcommand{\Z}{\mathbb Z}
\newcommand{\F}{\mathbb F}

\newcommand{\Bab}{B_{\mathrm{ab}}}
\newcommand{\Bna}{B_{\mathrm{nab}}}
\newcommand{\Nrm}{\mathsf N}
\title[Complete regular dessins: the nonmetacyclic case]{Regular dessins on $K_{2^e,2^f}$: the nonmetacyclic case}

\author[K. Hu]{Kan Hu}
\address{Department of Mathematics, Zhejiang Ocean University,
Zhoushan, Zhejiang 316022, P.R. China}
\email{hukan@zjou.edu.cn}

\date{}

\begin{document}
\begin{abstract}
We classify regular dessins on $K_{2^e,2^f}$ with nonmetacyclic
colour-preserving automorphism group, retaining the two specified
vertex rotations. We give nonredundant presentations and exact
counts in both the abelian and nonabelian Frattini cases.
For orientation- and colour-preserving isomorphisms, the number
is $2^{2e-2}+16$ when $e=f\ge4$ and
$3\cdot2^{2f-3}+16$ when $e>f\ge4$, with explicit counts for
all smaller exponents. The construction uses a parity criterion
in the abelian case and a uniquely determined odd quadratic
root with two lifts in the nonabelian case. A dessin in the
classification is reflexible exactly when its Frattini subgroup
is abelian. All symmetric classes are circular. We determine
every circular subfamily, including its face orders and genera;
the circular count is $2^e+8$ on the diagonal and
$2^{f-2}(2^{f-1}+1)+8$ off the diagonal for $e\ge f\ge4$.
\end{abstract}
\maketitle

\section{Introduction}
Cellular embeddings of complete bipartite graphs relate finite-group
factorisations to surface topology. Classifying those embeddings
whose symmetries act transitively on edges leads to a marked factorisation problem
for finite groups: one must classify the group together with two
cyclic subgroups and their specified generators. The generators
record the local rotations, so forgetting them loses information
about the embedding. In this paper we solve this problem for
$K_{2^e,2^f}$ when the colour-preserving automorphism group is
nonmetacyclic.

A \emph{dessin} is a connected bipartite graph, with a specified
vertex colouring, embedded cellularly in a closed oriented surface.
Its automorphisms preserve the colouring and orientation. An
automorphism fixing an edge fixes its ends and their cyclic orders;
connectedness then forces it to fix every edge. The automorphism
group therefore acts semiregularly on edges. A dessin is
\emph{regular} if this action is transitive, and \emph{complete}
if its underlying graph is complete bipartite.

A complete regular dessin on $K_{m,n}$ determines an
\emph{exact bicyclic triple}
\begin{equation}\label{exact}
 (G,a,b),\qquad G=\langle a\rangle\langle b\rangle,\qquad
 \langle a\rangle\cap\langle b\rangle=1,\qquad |a|=m,\quad |b|=n.
\end{equation}
Here $a,b$ are the rotations at the ends of a chosen edge.
Two such triples are equivalent if an isomorphism of groups sends
each specified generator to its counterpart. Equivalence classes
of triples correspond to orientation- and colour-preserving
isomorphism classes of complete regular dessins
\cite{FHNSW2020}. We recall the construction in
Section~\ref{dictionary}. The ordered pair records the two rotation
orders; the corresponding vertex classes have sizes $n,m$.

Two natural operations on dessins are reversal of orientation and
exchange of colours. The \emph{mirror} $\mathcal D^*$ reverses all
local cyclic orders, whereas the \emph{reciprocal} $\mathcal D'$
exchanges the colours. A dessin is \emph{reflexible} if it is
isomorphic to its mirror by an orientation-preserving,
colour-preserving map, equivalently if it admits an
orientation-reversing, colour-preserving self-symmetry. Otherwise
it is \emph{chiral}. It is \emph{symmetric} if it is isomorphic to
its reciprocal by an orientation-preserving map, and
\emph{circular} if every facial boundary walk is a simple cycle.
Symmetric complete regular dessins are the orientably regular
embeddings of complete bipartite graphs, with colour exchange
included among their map automorphisms.

The symmetric case was classified by Jones \cite{Jones2010}, building on
\cite{DJKNS2007,DJKNS2010,DJKNS2013,JNS2007,JNS2008,NSZ2002};
see also the survey \cite{Jones2015}. For dessins without an
assumed colour-exchanging symmetry, related results include the
uniqueness theorem of Fan and Li \cite{FL2018} and the
classification of circular dessins of odd prime-power bidegree
\cite{FLQ2018}, of $2$-power bidegree with metacyclic automorphism
group \cite{Fan2022}, and of coprime bidegree \cite{FLQ2023}.
Without assuming circularity, Hu, Nedela and Wang
\cite{HNW2019} classified complete regular dessins of odd
prime-power bidegree. Cai and Hu \cite{CH2024} treated the
metacyclic case for $K_{2^e,2^f}$.

The remaining case is specific to the prime two. A bicyclic
$p$-group is metacyclic when $p$ is odd, whereas bicyclic
$2$-groups need not be. Janko's structural classification
\cite{Janko2008} provides the group-theoretic starting point:
a nonmetacyclic bicyclic $2$-group has two metacyclic maximal
subgroups containing the cyclic factors and a third maximal
subgroup requiring three generators. This does not yet classify
exact bicyclic triples. One must determine which groups admit
disjoint cyclic factors of the prescribed orders, recover the
parameters from the specified generators, and distinguish
inequivalent markings of the same abstract group.

We carry out these steps by separating the two possibilities for
$\Phi(G)$. If $\Phi(G)$ is abelian, the squares $a^2,b^2$ form
a specified direct-product basis. We express the extension by
$G/\Phi(G)\cong C_2^2$ through six residues. Janko's
maximal-subgroup criterion identifies nonmetacyclicity with a
parity condition, after which the congruences give three explicit
families and complete boundary lists. If $\Phi(G)$ is nonabelian,
we impose exactness on Janko's presentations and obtain three
canonical families. Their remaining quadratic congruence has
one odd root modulo a half modulus and exactly two lifts to the
full modulus. In both cases we construct every group, prove
exhaustion, and recover all parameters intrinsically from its
marked triple.

The classification also yields a group-theoretic distinction
between reflection and chirality: a dessin in these families is
reflexible exactly when its Frattini subgroup is abelian.
The symmetric dessins all lie in the abelian case and are
circular. Circularity is more varied in the nonabelian case,
where longer simple facial cycles occur. We give explicit
circular subfamilies and enumerate them by face order.

Throughout, automorphisms act on the right:
\[
 u^{\alpha\beta}=(u^\alpha)^\beta,\qquad
 u^v=v^{-1}uv,\qquad [u,v]=u^{-1}v^{-1}uv.
\]
Thus $(u^v)^w=u^{vw}$ and $uv=vu[u,v]$.
The exponents $e,f$ are positive integers. Counts retain both
orientation and vertex colouring unless reciprocality or mirror
identification is explicitly stated.

\begin{theorem}[Classification and enumeration]\label{uniformmain}
Every nonmetacyclic exact triple $(G,a,b)$ of orders $2^e,2^f$
occurs in exactly one of the abelian Frattini lists in
Theorems~\ref{largefamilies} and \ref{boundary}, or the
nonabelian Frattini lists in Theorem~\ref{na:main}.
Each listed presentation defines a group of order $2^{e+f}$
with the prescribed disjoint cyclic factors. The lists are
nonredundant as marked triples.
For $e\ge f$, the number of orientation- and colour-preserving
isomorphism classes is
\[
 T(e,f)=
 \begin{cases}
 0,&f=1,\\
 1,&e=f=2,\\
 4,&f=2<e,\\
 16,&e=f=3,\\
 24,&f=3<e,\\
 2^{2e-2}+16,&e=f\ge4,\\
 3\cdot2^{2f-3}+16,&e>f\ge4.
 \end{cases}
\]
A listed dessin is reflexible if and only if $\Phi(G)$ is abelian.
The symmetric and circular classes are determined explicitly in
Theorems~\ref{uniformsymref} and \ref{uniformcircular}.
For $e<f$, the lists apply after exchanging the specified
generators and then restoring their original order.
\end{theorem}

The numbers $T(e,e)$ do not identify reciprocal dessins.
The reciprocal counts are $(T(e,e)+S(e))/2$, where
$S(2)=1$ and $S(e)=4$ for $e\ge3$; see
Corollary~\ref{uniformorbits}. This distinction is necessary on
the diagonal. For unequal exponents, placing the larger rotation
order first already chooses one representative up to reciprocality.

Section~\ref{dictionary} gives the group--dessin correspondence
and the required group-theoretic preliminaries.
Sections~\ref{abelian} and \ref{nonabelian} classify the two
Frattini cases. Sections~\ref{abelianexternal} and
\ref{nonabelianexternal} determine their faces, external
symmetries and circular subfamilies. Section~\ref{uniformcount}
assembles the counts and completes the proof of
Theorem~\ref{uniformmain}. Section~\ref{stabilisation} explains the precise scope of stabilisation
in the larger exponent, compares the odd-prime and $2$-group counts,
and identifies the compensating family weights. The appendices
contain reproducible checks of the parameter lists and their
geometric invariants.

\section{Preliminaries}\label{dictionary}
In this section, we explain the correspondence between complete regular dessins
and exact $(m,n)$-bicyclic triples, and then focus on the special case $m=2^e$ and $n=2^f$.
\subsection{Correspondence}
\begin{proposition}[{\cite{FHNSW2020,Fan2021}}]\label{dessindictionary}
The isomorphism classes of regular dessins $\mathcal{D}$ with underlying graph $K_{m,n}$
correspond bijectively to the equivalence classes of exact $(m,n)$-bicyclic triples $(G,a,b)$.
Moreover,
\begin{enumerate}[label=\rm(\alph*)]
  \item the dessin $\mathcal{D}$ is reflexible if and only if $G$ has an automorphism $a\to a^{-1}$ and $b\to b^{-1}$;
  \item the dessin $\mathcal{D}$ is symmetric if and only if $G$ has an automorphism transposing $a$ and $b$;
  \item the dessin $\mathcal{D}$ is circular if and only if $\langle a\rangle\cap\langle ab\rangle=\langle b\rangle\cap\langle ab\rangle=1$.
\end{enumerate}
\end{proposition}
\begin{proof}
Start with an exact bicyclic triple $(G,a,b)$ as given in~\eqref{exact}, put $A=\langle a\rangle$ and $B=\langle b\rangle$,
and label the edges by the elements of $G$. The black vertices are the
right cosets $Ag$ and the white vertices are the right cosets $Bg$.
Edge $g$ joins $Ag$ to $Bg$. At vertex $Ag$, prescribe the cyclic order
\[
 g,\ ag,\ a^2g,\ldots,a^{m-1}g,
\]
and prescribe the analogous order at $Bg$ using $b$. These orders do not
depend on the choice of representative of the coset: replacing $g$ by
$a^i g$ merely starts the same black cyclic order at a different place.

Any two cosets $Ag$ and $Bh$ meet in at most one element. Indeed,
if $z,z'$ belong to both, then $z'z^{-1}$ lies in both $A$ and $B$,
so $z=z'$. They also meet in at least one element: the condition
$a^i g=b^j h$ is equivalent to $g h^{-1}=a^{-i}b^j$, and the latter
is possible because $G=AB$. Hence the resulting graph is complete
bipartite, with $|G|/|A|=n$ black vertices and $m$ white vertices.
The cyclic orders define an oriented cellular embedding by the usual
ribbon construction: thicken the vertices to oriented discs, attach an
edge band in the specified cyclic position, and cap every boundary
component by a disc; see~\cite{GT1987} for explicit details.

Right multiplication acts on edge labels by $g^k=gk$. It preserves
both coset systems and both cyclic orders, and is regular on the edges.
The full colour-preserving, orientation-preserving automorphism group
cannot be larger, since an automorphism is determined by the image
of one edge. Thus this group is $G$, acting on the right.

Conversely, choose an edge of a complete regular dessin $\mathcal{D}$
with underlying graph $K_{m,n}$, and identify
its automorphism group with the edge set by this regular action.
The stabilizers of its black and white ends are cyclic: their actions
on the incident edges preserve a cyclic order, and are faithful.
Choose $a,b$ to advance these two cyclic orders by one place.
They have the stated orders, and their intersection fixes the chosen
edge, so is trivial. The group has as many elements as the graph has
edges, namely $mn$. Therefore
$|\langle a\rangle\langle b\rangle|=mn=|G|$, giving \eqref{exact}.

Changing the chosen edge conjugates both distinguished generators by
the same group element, which gives an isomorphic exact triple.
An isomorphism of dessins transports the two distinguished rotations,
and an equivalence of triples transports the cosets and cyclic orders
in the construction. This proves the asserted bijection.

Reversing the orientation reverses both cyclic orders, so the mirror
dessin is represented by $(G,a^{-1},b^{-1})$. It is isomorphic to
the original precisely when the indicated generator-inverting
automorphism exists; this proves (a). Exchanging the two colors
similarly replaces the triple by $(G,b,a)$, proving (b).

For (c), write $c=ba$. A face boundary has black vertices
$A c^i$ and white vertices $B a c^i$, for $0\le i<|c|$:
on each pair of successive sides the edge label advances from
$c^i$ to $ac^i$ and then to $bac^i=c^{i+1}$.
Two black vertices coincide exactly when a nonidentity power
$c^{i-j}$ lies in $A$. Two white vertices coincide exactly when
$a c^{i-j}a^{-1}$ lies in $B$. Since
$a(ba)a^{-1}=ab$ and $a$ normalizes $A$, these are respectively
the conditions that a nonidentity power of $ab$ lies in $A$
or in $B$. The two colors prevent coincidences between a black
and a white vertex. Thus all vertices of a face boundary are
distinct precisely when both intersections in (c) are trivial.
Regularity makes the same condition apply to every face.
\end{proof}
In the rest of the paper, we concentrate on complete regular dessins
with underlying graph $K_{2^e,2^f}$, which corresponds to exact bicyclic triples $(G,a,b)$ with
\begin{equation}\label{2exact}
 G=\langle a\rangle\langle b\rangle,\qquad
 \langle a\rangle\cap\langle b\rangle=1,\qquad
 |a|=2^e,\quad |b|=2^f,
\end{equation}
where $e,f$ are positive integers. We order $e\ge f$ when listing the
nonabelian families and the final counts; the abelian congruences
are valid without an ordering assumption.

\subsection{Some results on \texorpdfstring{$p$}{p}-groups}
Let $G$ be a finite $p$-group of exponent $p^h$. The series
\[
1=\Omega_0(G)\leq\Omega_1(G)\leq \cdots\leq\Omega_h(G)=G
\]
and
\[
G=\mho_0(G)\geq\mho_{1}(G)\geq \cdots\geq\mho_h(G)=1
\]
are called the \textit{upper power series} and the \textit{lower power series} of $G$, respectively, where
\[
\mho_i(G)=\langle g^{p^i}\mid g\in G\rangle\quad\text{and}\quad\Omega_i(G)=\langle g\in G\mid g^{p^i}=1\rangle,\qquad 0\leq i\leq h.
\]
The \textit{Frattini subgroup} $\PhiG$ of $G$ is the intersection of the maximal subgroups of $G$.
We recall
\begin{equation}\label{frattinifact}
 \Phi(G)=\mho_1(G)G',\qquad G/\Phi(G)\cong C_p^{\,d(G)},
\end{equation}
where $d(G)$ is the minimum number of generators.
The Burnside basis theorem says that elements generate
$G$ if and only if their images span $G/\Phi(G)$~\cite{Huppert1967}.

Recall that a group $G$ is \textit{metacyclic} if it contains a cyclic normal subgroup $N$ such that $G/N$ is cyclic; otherwise, it is \textit{nonmetacyclic}.
\begin{lemma}\label{primitive}
Let $G$ be a finite two-generator noncyclic $2$-group. Then $G$ is metacyclic if and only if
$G'\le\langle g^2\rangle$ for some $g\in G\setminus\PhiG$.
\end{lemma}
\begin{proof}
If $G$ is metacyclic, then $G$ has a cyclic normal subgroup $C=\langle g\rangle$ with cyclic quotient.
Then $G'\le C$. If $C\le\PhiG$, then $G/\PhiG$, being a quotient of the cyclic group $G/C$,
would be cyclic, a contradiction. Hence a generator $g$ of $C$ lies outside $\PhiG$, and
consequently $G'\le C\cap\Phi(G)=\langle g^2\rangle$.

Conversely, suppose $g\notin\PhiG$ and $G'\le\langle g^2\rangle$. Then $\langle g\rangle$ is normal.
We may complete the nonzero image of $g$ to a basis of $G/\PhiG$ and then choose a lift $h$ of the other basis vector.
The Burnside basis theorem implies $G=\langle g,h\rangle$. Therefore $G/\langle g\rangle$ is cyclic, generated by the image of $h$.
\end{proof}

\begin{proposition}[{\cite[Theorem 5.1]{Janko2008}}]\label{prop:janko}
Let $G=AB$ be a nonmetacyclic $2$-group, where the subgroups $A$ and $B$
are cyclic. Let $\{U,V,W\}$ be the set of maximal subgroups of $G$, where $A<U$ and $B<V$.
Then $U$ and $V$ are metacyclic, and the remaining maximal subgroup $W$
has rank $d(W)=3$.
\end{proposition}

\begin{lemma}\label{squarebasis}
Let $(G,a,b)$ be an exact $(2^e,2^f)$-bicyclic triple with $e,f\ge1$. Then $G$ is noncyclic, $G/\PhiG\cong C_2^2$, and
$ \PhiG=\langle a^2\rangle\langle b^2\rangle$ is an exact product, with specified factors of orders $2^{e-1}$ and $2^{f-1}$.
\end{lemma}
\begin{proof}
Since the factors $\langle a\rangle$ and $\langle b\rangle$ of $G$ are disjoint nontrivial cyclic $2$-subgroups,
$G$ is not cyclic, and the Burnside basis theorem gives $G/\PhiG\cong C_2^2$ so that the images of $a,b$ form a basis.
An element $a^i b^j\in G$ belongs to $\PhiG$ if and only if both $i$ and $j$ are even. Hence the underlying set of $\PhiG$ is precisely the set $\langle a^2\rangle\langle b^2\rangle$. The factors have trivial intersection because the original factors do, and their orders are the stated orders. In particular their product is a subgroup.
\end{proof}

\begin{lemma}\label{kernelmeta}
Let $G$ be an exact bicyclic $2$-group. Then $\PhiG$ is metacyclic.
\end{lemma}
\begin{proof}
If $G$ is metacyclic, then $\PhiG$ is clearly metacyclic.
If $G$ is nonmetacyclic, by Proposition~\ref{prop:janko}, $G$ contains exactly one nonmetacyclic maximal subgroup, and the remaining
two maximal subgroups of $G$ are metacyclic. The assertion follows from the facts that
$\PhiG$ is the intersection of all maximal subgroups of $G$
and every subgroup of a metacyclic group is metacyclic.
\end{proof}

For an element $g\in G$ of order $2^t>1$, the unique involution $g^{2^{t-1}}$ contained in $\langle g\rangle$
is its \textit{terminal involution}.
\begin{lemma}[Terminal involutions]\label{terminal}
Two nontrivial cyclic $2$-subgroups in a group $G$ have
trivial intersection if and only if their terminal involutions
differ.
\end{lemma}
\begin{proof}
Every nontrivial subgroup of a cyclic $2$-group contains its unique
involution. A nontrivial intersection therefore forces the two
terminal involutions to agree, and agreement itself supplies a
nontrivial intersection.
\end{proof}

The following result is well known.
\begin{lemma}\label{roots}
For $L=2^h$ with $h\ge3$, the roots of $r^2=1\pmod L$ are
\[
 r=-1,\quad -1+L/2,\quad 1,\quad 1+L/2.
\]
Modulo two there is only the root one; modulo four the roots are one
and three.
\end{lemma}
\begin{proof}
Any root is odd. The even integers $r-1$ and $r+1$ differ by two,
so one has exactly one factor of two. The divisibility
$2^h\mid(r-1)(r+1)$ forces the other to be divisible by
$2^{h-1}$. Hence $r\equiv1$ or $-1\pmod{2^{h-1}}$.
Each congruence gives two residues modulo $2^h$, namely the four
listed roots. Direct squaring verifies them, and they are distinct
for $h\ge3$. The two smaller moduli are immediate.
\end{proof}
\subsection{Power and norm formulas}
For an odd integer $Q$ and an integer $j\ge0$, write
\begin{equation}\label{na:norm}
 \Nrm_j(Q)=1+Q+\cdots+Q^{j-1},\qquad \Nrm_0(Q)=0.
\end{equation}
\begin{lemma}[Geometric sums at powers of two]\label{na:normlemma}
Let $q\equiv3\pmod4$ be an odd integer, and write
$\lambda=v_2(q+1)$. If $q=-1$, interpret the even sums below as zero.
For $j\ge1$, when the sums are nonzero,
\begin{align}
 v_2(q^{2^j}-1)&=\lambda+j,\label{na:qval}\\
 v_2(\Nrm_{2^j}(q))&=\lambda+j-1.\label{na:normval}
\end{align}
If $t\equiv1\pmod4$, then
\begin{equation}\label{na:positiveval}
 v_2(\Nrm_{2^j}(t))=j.
\end{equation}
The same formula holds for $t=1$, since the sum equals $2^j$.
For positive odd $i$, replacing $q$ by $q^i$ does not change $\lambda$.
\end{lemma}
\begin{proof}
Factor
\[
 q^{2^j}-1=(q-1)(q+1)
          (q^2+1)(q^4+1)\cdots(q^{2^{j-1}}+1).
\]
The first factor has valuation one, the second has valuation
$\lambda$, and each subsequent factor has valuation one because
an odd square is one modulo eight.
There are $j-1$ subsequent factors. Adding the valuations gives
$\lambda+j$. Dividing by $q-1$, whose valuation is one, proves
\eqref{na:normval}.

For $t\equiv1\pmod4$, the same factorisation gives
$v_2(t^{2^j}-1)=v_2(t-1)+j$, because $v_2(t+1)=1$.
Division by $t-1$ gives \eqref{na:positiveval}.
Finally
$(q^i+1)/(q+1)=q^{i-1}-q^{i-2}+\cdots-q+1$ is odd for odd $i$,
so $v_2(q^i+1)=v_2(q+1)$.
The assertions for the zero or constant sums follow directly.
\end{proof}

\begin{lemma}[Right-action power formulas]\label{na:power}
If $E$ is an abelian normal subgroup, $x\in G$, and $d\in E$, then
\begin{equation}\label{na:powerformula}
 (xd)^j=x^j\prod_{i=0}^{j-1}d^{x^i}.
\end{equation}
If $[r,s]=v\in E$ and conjugation by $r$ induces a right action
$\rho$ on $E$,
then
\begin{equation}\label{na:commnorm}
 [r^j,s]=\prod_{i=0}^{j-1}v^{\rho^i},\qquad
 sr^j=r^js[r^j,s]^{-1}.
\end{equation}
\end{lemma}
\begin{proof}
The identity $dx=x d^x$ moves a kernel element to the right.
Assuming \eqref{na:powerformula} for $j$, multiplication by $xd$
conjugates its accumulated kernel product by $x$ and adds $d$.
This gives the product for $j+1$; its factors commute inside $E$.
For \eqref{na:commnorm}, use $r^s=rv$:
\[
 (r^j)^s=(rv)^j
       =r^j\prod_{i=0}^{j-1}v^{r^i}.
\]
By definition this equals $r^j[r^j,s]$.
The final relation follows from
$r^js=sr^j[r^j,s]$.
All conjugations in these calculations are right conjugations.
\end{proof}

\section{The abelian Frattini case}\label{abelian}
If $\min(e,f)=1$, one factor has index two and the group is
metacyclic. Thus the nonmetacyclic classification has no cases in
that range. From now on assume $e,f\ge2$ and write $M:=2^{e-1}$ and $N:=2^{f-1}$, unless explicitly
stated otherwise. Throughout Section~\ref{abelian}, assume that
\[
 K=\PhiG=\langle x\rangle\times\langle y\rangle\cong C_M\times C_N.
\]
\subsection{Constructing the abelian Frattini extension}\label{extension}
Conjugation on $K$ gives right automorphisms
\[
 u^\alpha=u^a,\qquad u^\beta=u^b \quad (u\in K).
\]
As $a^2=x$ and $b^2=y$ belong to the abelian group $K$,
\[
 u^{\alpha^2}=u^{a^2}=u,\qquad
 u^{\beta^2}=u^{b^2}=u.
\]
Also $G/K$ is abelian, so $ab$ and $ba$ differ by an element of $K$.
Conjugation by such an element is trivial on $K$, whence
$\alpha\beta=\beta\alpha$.
In particular, although we use right conjugation, on this kernel
\begin{equation}\label{lefttoright}
 a u a^{-1}=u^{a^{-1}}=u^\alpha,\qquad
 b u b^{-1}=u^{b^{-1}}=u^\beta.
\end{equation}
The last equalities use involutivity, not a change of convention.

Define
\begin{equation}\label{tdef}
 t=ba(ab)^{-1}\in K,\qquad ba=t\,ab.
\end{equation}
We call $t$ the \emph{collection coefficient}. With
our commutator convention the commutator is instead
\begin{equation}\label{actualcomm}
 [b,a]=(ab)^{-1}ba=(ab)^{-1}t(ab)=t^{ab}
       =t^{\alpha\beta}.
\end{equation}
Thus $t$ and $[b,a]$ have the same normal closure, but should not be
identified as elements.

\begin{theorem}[Extension conditions and normal forms]\label{construction}
Let $K=\langle x\rangle\times\langle y\rangle\cong C_M\times C_N$,
let $\alpha,\beta\in\Aut(K)$ be right automorphisms, and let $t\in K$.
There is an exact triple with $a^2=x$, $b^2=y$, $u^a=u^\alpha$,
$u^b=u^\beta$ and $ba=t\,ab$ if and only if
\begin{align}
 x^\alpha&=x,&y^\beta&=y,\label{E1}\\
 \alpha^2&=\beta^2=\id,&\alpha\beta&=\beta\alpha,\label{E3}\\
 x^\beta&=x\,t\,t^\alpha,&
 y^\alpha&=y\,t^{-1}(t^{-1})^\beta.\label{E2}
\end{align}
When these conditions hold, the presentation consisting of the relations
of $K$ and
\begin{equation}\label{eqn:relpres}
 a^2=x,\quad b^2=y,\quad
 u^a=u^\alpha,\quad u^b=u^\beta\ (u\in K),\quad ba=t\,ab
\end{equation}
has exactly the $4MN$ distinct normal forms
\begin{equation}\label{normal}
 u a^\epsilon b^\eta,\qquad u\in K,\quad \epsilon,\eta\in\{0,1\}.
\end{equation}
Its Frattini subgroup is precisely the embedded copy of $K$.
\end{theorem}
\begin{proof}
For necessity, \eqref{E1} follows because an element centralizes its
own square. We established \eqref{E3} above. The collection relation
gives $bab^{-1}=ta$ and $aba^{-1}=t^{-1}b$.
Using \eqref{lefttoright} to move elements of $K$ to the left gives
\[
 x^\beta=ba^2b^{-1}=(ta)^2=t\,t^\alpha x,\qquad
 y^\alpha=ab^2a^{-1}=(t^{-1}b)^2=t^{-1}(t^{-1})^\beta y.
\]
This proves \eqref{E2}.

For sufficiency, we verify normal forms rather than merely assuming
that the presentation retains a copy of $K$. Take the elements of $K$
as letters, with their multiplication table as relations, and use
the collection rules
\begin{equation}\label{rules}
 a u\longrightarrow u^\alpha a,\quad
 b u\longrightarrow u^\beta b,\quad
 aa\longrightarrow x,\quad bb\longrightarrow y,\quad
 ba\longrightarrow t\,ab.
\end{equation}
Adjacent kernel letters are multiplied, and identity letters deleted.
These rules terminate. A useful decreasing lexicographic measure is:
the number of $a,b$ letters; the number of pairs in which a $b$ precedes
an $a$; the sum, over kernel letters, of the number of $a,b$ letters
to their left; and finally the number of kernel letters.
The rules for $aa,bb$ decrease the first entry; the rule for $ba$
decreases the second; moving a kernel letter left decreases the third;
and combining kernel letters decreases the third or fourth.
The irreducible words are exactly \eqref{normal}.

We check every possible overlap of rules. Overlaps of three kernel
letters agree by associativity in $K$. For $a u v$ and $b u v$,
agreement is exactly the assertion that $\alpha,\beta$ preserve
multiplication. For $aa u$ and $bb u$, the two reductions agree because
$\alpha^2=\beta^2=\id$ and $K$ is abelian. For $aaa$ and $bbb$,
agreement is respectively $x^\alpha=x$ and $y^\beta=y$.
There are three remaining types:
\[
\begin{array}{c|c|c}
\text{word}&\text{one reduction}&\text{other reduction}\\ \hline
ba u&t\,u^{\alpha\beta}ab&u^{\beta\alpha}t\,ab\\
baa&t\,t^\alpha x\,b&x^\beta b\\
bba&y\,a&t^\beta t\,y^\alpha a .
\end{array}
\]
The first row agrees by commutativity of $K$ and of the two actions.
The second is the first equation of \eqref{E2}.
The second equation of \eqref{E2} says
$y^\alpha=y\,t^{-1}(t^\beta)^{-1}$, which gives the third row.
This exhausts the overlaps, since every rule has a left side of
length two.

The elementary termination-and-overlap argument now gives uniqueness
of reduction: if two first reductions are disjoint, they commute;
if they overlap, the checks just made join their results.
Induction on the decreasing measure joins any two full reduction
sequences. Consequently each presented element has one normal form,
and no two of the words \eqref{normal} are identified.
The presented monoid is a group: kernel letters are invertible,
$a$ has inverse $x^{-1}a$, and $b$ has inverse $y^{-1}b$.
Hence the presentation defines a group of order $4MN$.

The orders of $a,b$ are $2M,2N$, since their squares have orders $M,N$
and their images modulo $K$ have order two.
If $a^i=b^j$, the four normal-form cosets force both $i,j$ to be even.
The equality then lies in
$\langle x\rangle\cap\langle y\rangle=1$, so the two cyclic subgroups
intersect trivially. Their product has $4MN$ elements and is all of
the group. Also $K=\langle a^2,b^2\rangle$, so $a,b$ generate $G$.
The quotient $G/K$ is elementary abelian of order four.
Thus $\Phi(G)=G^2G'\le K$, while $K\le\Phi(G)$ because it is generated
by squares. The two inclusions give equality.
\end{proof}

\subsection{The six-residue classification}\label{sixsection}
Let $(\alpha,\beta,t)$ satisfy the extension conditions of Theorem~\ref{construction}. Since $t\in K$, $t=x^py^q$ for some $p\in\Z_M$ and
$q\in\Z_N$. The identities \eqref{E1}, \eqref{E2} and \eqref{E3} in Theorem~\ref{construction} imply that
\[
x^{\alpha}=x,\quad y^{\alpha}=x^Py^{Q},\quad x^{\beta}=x^Ry^S,\qquad y^{\beta}=y
\]
for some $P,R\in\Z_M$ and $Q,S\in\Z_N$. Thus, for any $(u,v)\in\Z_M\times\Z_N$,
\begin{align*}
(x^uy^v)^{\alpha}&=(x^{\alpha})^u(y^{\alpha})^v=x^u(x^Py^{Q})^v=x^{u+Pv}y^{Qv},\\
(x^uy^v)^{\beta}&=(x^{\beta})^u(y^{\beta})^v=(x^Ry^S)^uy^v=x^{Ru}y^{Su+v}.
\end{align*}
We may write $K$ and its basis $\{x,y\}$ additively as
\[
 K=\Z_M\times\Z_N,\qquad x=(1,0),\quad y=(0,1),
\]
so that the above equations are transformed into the following form:
\begin{equation}\label{coaction}
  t=(p,q),\qquad (u,v)^\alpha=(u+Pv,Qv),\qquad
 (u,v)^\beta=(Ru,Su+v).
\end{equation}
\begin{theorem}[Six-residue presentation]\label{six}
Exact $(2^e,2^f)$-bicyclic triples $(G,a,b)$ with abelian Frattini subgroup are in bijective correspondence
with the six-residue data $P,R,p\in\Z_M$ and $Q,S,q\in\Z_N$ satisfying the following conditions:
\begin{align}
 NP&=0\pmod M,& MS&=0\pmod N,\label{A1}\\
 P(1+Q)&=0\pmod M,& Q^2&=1\pmod N,\label{A2}\\
 R^2&=1\pmod M,& S(1+R)&=0\pmod N,\label{A3}\\
 PS&=0\pmod M,& PS&=0\pmod N,\label{A4}\\
 P(R-1)&=0\pmod M,& S(Q-1)&=0\pmod N,\label{A5}\\
 R&=1+2p+Pq\pmod M,& S&=(1+Q)q\pmod N,\label{A6}\\
 P&=-(1+R)p\pmod M,& Q&=1-Sp-2q\pmod N.\label{A7}
\end{align}
In particular $Q,R$ are odd. The group has the marked presentation
\begin{equation}\label{overviewpresentation}
 \begin{split}
 G=\langle a,b,x,y\mid{}&
 x^M=y^N=[x,y]=1,\quad a^2=x,\quad b^2=y,\\
 &x^a=x,\quad y^a=x^P y^Q,\quad
 x^b=x^R y^S,\quad y^b=y,\quad ba=x^p y^q ab\rangle.
 \end{split}
\end{equation}
This theorem allows arbitrary residues $p,q$; the restriction to
odd $p,q$ will select the nonmetacyclic triples in
Proposition~\ref{parity}. Distinct admissible six-residue data
define nonisomorphic marked triples.

\end{theorem}
\begin{proof}
We use the identities \eqref{E1}, \eqref{E2} and \eqref{E3}, which guarantee
the admissibility of the datum, to prove this result.

First the two maps must be well defined on the indicated residue
classes. Replacing $v$ by $v+N$ in the first map changes its first
coordinate by $NP$, which must vanish modulo $M$. Replacing $u$
by $u+M$ in the second changes its second coordinate by $MS$,
which must vanish modulo $N$. All other changes are already
multiples of their own coordinate modulus. Thus well-definedness
is equivalent to \eqref{A1}.
Squaring the right action $\alpha$ gives
\[ (u,v)^{\alpha^2}=(u+P(1+Q)v,Q^2v),
\]
so $\alpha^2=\id$ is equivalent to \eqref{A2}.
Likewise
\[
 (u,v)^{\beta^2}=(R^2u,S(1+R)u+v),
\]
giving \eqref{A3}. An endomorphism whose square is the identity is
automatically an automorphism, with itself as inverse.
Thus these equations also justify the required invertibility.

To check the composition order explicitly, our right-action convention
gives
\begin{align*}
 (u,v)^{\alpha\beta}
   &=(Ru+RPv,\ Su+(SP+Q)v),\\
 (u,v)^{\beta\alpha}
   &=((R+PS)u+Pv,\ QSu+Qv).
\end{align*}
Equating the coefficients of $u,v$ in the appropriate coordinates
gives exactly \eqref{A4} and \eqref{A5}.

In additive notation \eqref{E2} reads
$x^\beta=x+t+t^\alpha$ and $y^\alpha=y-t-t^\beta$.
Since
\[
 t^\alpha=(p+Pq,Qq),\qquad t^\beta=(Rp,Sp+q),
\]
comparison of the two coordinates gives \eqref{A6} and \eqref{A7}.
Theorem~\ref{construction} therefore proves both necessity and
sufficiency of the entire list, and proves that \eqref{overviewpresentation}
has the required order and exact factorisation.

Finally, in a marked triple the elements $x=a^2,y=b^2$ are fixed by
the marking. The coordinates of $y^a$ recover $P,Q$, those of $x^b$
recover $R,S$, and the equation $ba=t\,ab$ recovers $p,q$.
A marked isomorphism preserves all these coordinates. Conversely,
equal data give the same presentation. This proves nonredundancy.
\end{proof}
\subsection{The parity criterion}\label{paritysection}
The following lemma identifies the metacyclic groups among those constructed above.
\begin{lemma}\label{dsquares}
For an admissible six-residue datum, the derived subgroup is
\begin{equation}\label{derived}
 D=G'=\langle t,t^\alpha,t^\beta,t^{\alpha\beta}\rangle\le K.
\end{equation}
For $w\in K$, the squares in the three nontrivial Frattini cosets are
\begin{align}
 (wa)^2&=w+w^\alpha+x,\label{squarea}\\
 (wb)^2&=w+w^\beta+y,\label{squareb}\\
 (wab)^2&=w+w^{\alpha\beta}+z,\qquad
 z=x+y+t^\alpha.\label{squareab}
\end{align}
Moreover $G$ is metacyclic if and only if $D$ is contained in the cyclic
subgroup generated by at least one of these squares, for some $w$.
\end{lemma}
\begin{proof}
In any two-generator group the normal closure of $[b,a]$ is the derived
subgroup: killing it makes the two generators commute, whereas every
conjugate of it already lies in $G'$.
By \eqref{actualcomm}, its normal closure equals that of $t$.
Conjugation by $K$ is trivial on $K$, and conjugation by $a,b$
induces the commuting involutions $\alpha,\beta$.
The possible conjugates of $t$ are therefore the four elements displayed
in \eqref{derived}; their subgroup is normal and is exactly $G'$.

For the first square use $a w=w^\alpha a$, which follows from
\eqref{lefttoright}. Then
$(wa)^2=w w^\alpha a^2$, giving \eqref{squarea}; the second is the
same calculation with $b$. Next,
\[
 (ab)^2=a(ba)b=a\,t\,a b^2=t^\alpha a^2b^2.
\]
Conjugation by $(ab)^{-1}$ on $K$ is also $\alpha\beta$, because
these involutions commute. Hence
$ab\,w=w^{\alpha\beta}ab$, and the third square formula follows.
Every element outside $K$ has one of these three normal-form shapes.
The last assertion is now precisely Lemma~\ref{primitive}.
\end{proof}

The cyclic-subgroup test above is finite, but the following result
removes it completely from the nonmetacyclic classification.

\begin{proposition}\label{parity}
An admissible six-residue datum gives a nonmetacyclic group if and only
if both $p$ and $q$ are odd.
\end{proposition}
\begin{proof}
The root equations make $Q,R$ odd; the norm equations then make
$P,S$ even. Hence both actions are trivial on
$K/2K=\langle X,Y\rangle\cong C_2^2$, where $X,Y$ are the
images of $x,y$. Put $H=2K$. It is characteristic in $K$ and
normal in $G$. The quotient $G/H$ has order sixteen and its
subgroup $K/H$ is central. If $t$ has image $pX+qY$, then
\[
 (ab)^2H=(1+p)X+(1+q)Y.
\]

The third maximal subgroup is $W=\langle ab,K\rangle$;
the other two contain $a$ and $b$, respectively.
The quotient $W/H$ is abelian of order eight, generated by the
two central involutions $X,Y$ and $abH$. Moreover
$H=\Phi(K)\le\Phi(W)$ because squares and commutators of a
subgroup lie among those of the ambient group. Thus $d(W)=3$
exactly when $(ab)^2H=1$, and otherwise $d(W)=2$.
This condition is precisely that $p,q$ are odd.

If $G$ is nonmetacyclic, Proposition~\ref{prop:janko} gives
$d(W)=3$. Conversely $d(W)=3$ rules out metacyclicity of $G$:
every subgroup of a metacyclic group is metacyclic and has at most
two generators, by intersecting it with a cyclic normal subgroup
and lifting a generator of its cyclic image. This proves both
directions and all four parity patterns at once.
\end{proof}

\begin{remark}
The parity criterion concerns the coordinates of $t$.
It can also be phrased using the coordinates of $[b,a]$, since the
actions are trivial on $K/\mho_1(K)$ and \eqref{actualcomm} gives the same
coordinate parities. The two elements themselves can nevertheless
differ, so their actual coordinates must not be interchanged in
presentations or face calculations.
\end{remark}

\subsection{Solving the nonmetacyclic congruences}\label{arithmetic}
All data in this section have $p,q$ odd.
We first eliminate two of the six residues.

\begin{lemma}[Four-residue reduction]\label{four}
Choose odd $p\in\Z_M,q\in\Z_N$, and set
\begin{equation}\label{PS}
 P=-(1+R)p\pmod M,\qquad S=(1+Q)q\pmod N.
\end{equation}
The six-residue conditions are equivalent to
\begin{align}
 R^2&=1\pmod M,&Q^2&=1\pmod N,\label{F1}\\
 N(1+R)&=0\pmod M,&M(1+Q)&=0\pmod N,\label{F2}\\
 (1+R)(1+Q)&=0\pmod{\max(M,N)},\label{F3}\\
 (1+R)(1+pq)&=2(1+p)\pmod M,\label{F4}\\
 (1+Q)(1+pq)&=2(1-q)\pmod N.\label{F5}
\end{align}
Every admissible quadruple $(R,Q,p,q)$ defines one nonmetacyclic marked
triple, and different quadruples define different marked triples.
\end{lemma}
\begin{proof}
The root equations for $Q,R$ in \eqref{A2},\eqref{A3} give
\eqref{F1}.
Substitute \eqref{PS} into \eqref{A1}. Multiplication by the odd integers
$p,q$ is invertible modulo either power of two, so these become
\eqref{F2}.
The two equations $PS=0$ become
$(1+R)(1+Q)pq=0$ modulo $M$ and modulo $N$.
Again $pq$ is odd and invertible. Since one of $M,N$ divides the other,
the two equations are exactly \eqref{F3}.

The remaining square-action equations
$P(1+Q)=0$ and $S(1+R)=0$ follow from \eqref{F3}.
The commutation equations in \eqref{A5} also follow, because
\[
 P(R-1)=-(R^2-1)p,\qquad S(Q-1)=(Q^2-1)q.
\]
The first equation of \eqref{A6}, after substituting $P$, becomes
\[
 R=1+2p-(1+R)pq\pmod M.
\]
Adding one and rearranging gives \eqref{F4}.
The second equation of \eqref{A7}, after substituting $S$, becomes
$Q=1-(1+Q)pq-2q\pmod N$, which is \eqref{F5}.
The other equations of \eqref{A6},\eqref{A7} are precisely the
definitions \eqref{PS}. We have accounted for every six-residue equation,
so reversing the calculation proves sufficiency.

All congruences involving residues from the other modulus are
well defined under \eqref{F2}. For example, changing $q$ by $N$
changes the left side of \eqref{F4} by $(1+R)pN$, a multiple of $M$.
Changing $R$ by $M$ in \eqref{F3} changes its left side by
$M(1+Q)$, divisible by both moduli by \eqref{F2}.
The other checks are identical. Thus representative choices introduce
no extra parameters. Nonmetacyclicity follows from
Proposition~\ref{parity}, and uniqueness follows from Theorem~\ref{six}.
\end{proof}

\begin{theorem}[Explicit nonredundant families]\label{largefamilies}
Suppose $M,N\ge8$, equivalently $e,f\ge4$.
All nonmetacyclic exact bicyclic triples with abelian Frattini subgroup are
given by the following three disjoint families.
In each family $\rho,\sigma,\lambda,\mu$ range over $\{0,1\}$ where
they occur, and $P,S$ are recovered by \eqref{PS}.
\begin{enumerate}[label=\textup{(\Roman*)},leftmargin=2.5em]
\item For all $M,N$, take
\[
 R=-1+\rho M/2,\quad Q=-1+\sigma N/2,\quad
 p=-1+\lambda M/2,\quad q=1+\mu N/2.
\]
All sixteen choices are admissible.
\item Take
\[
 R=1+\rho M/2,\qquad Q=-1+\sigma N/2,
\]
with $p$ any odd residue modulo $M$, and require
\[
 q=1\pmod{\max(M,N)/2}.
\]
This family occurs exactly when $M\le2N$.
If $M\le N$, then $q=1$ or $1+N/2$, and $\sigma$ is arbitrary.
If $M=2N$, then $q=1$ and $\sigma=0$.
\item Take
\[
 R=-1+\rho M/2,\qquad Q=1+\sigma N/2,
\]
with $q$ any odd residue modulo $N$, and require
\[
 p=-1\pmod{\max(M,N)/2}.
\]
This family occurs exactly when $N\le2M$.
If $N\le M$, then $p=-1$ or $-1+M/2$, and $\rho$ is arbitrary.
If $N=2M$, then $p=-1$ and $\rho=0$.
\end{enumerate}
The congruences in the two latter families refer to representatives
of the relevant residues. The explicit alternatives just stated
specify those residues unambiguously; in particular, $p=-1$ means
the residue $M-1$, not a restriction on a preferred integer lift.
\end{theorem}
\begin{proof}
By Lemma~\ref{roots}, write
$R=\pm1+\rho M/2$ and $Q=\pm1+\sigma N/2$.
Since $p,q$ are odd, $1+pq$ is even. Thus every term
$(M/2)(1+pq)$ vanishes modulo $M$, and every term
$(N/2)(1+pq)$ vanishes modulo $N$.
Equation \eqref{F4} accordingly simplifies as follows:
\[
\begin{array}{c|c}
R=-1+\rho M/2&2(1+p)=0\pmod M\\
R=1+\rho M/2&2p(q-1)=0\pmod M .
\end{array}
\]
Because $p$ is invertible modulo $M$, these conditions mean respectively
$p=-1\pmod{M/2}$ and $q=1\pmod{M/2}$.
Likewise \eqref{F5} becomes
\[
\begin{array}{c|c}
Q=-1+\sigma N/2&q=1\pmod{N/2}\\
Q=1+\sigma N/2&p=-1\pmod{N/2},
\end{array}
\]
where in the second row we cancel the odd factor $q$.

There are four sign combinations to consider.
If both signs are positive, each of $1+R,1+Q$ is twice an odd
integer: $M/2,N/2$ are divisible by four.
Their product is divisible by four but not by eight.
This contradicts \eqref{F3}, since $\max(M,N)\ge8$.
Thus there is no positive-positive family.

If both signs are negative, the norm equations give
$p=-1+\lambda M/2$ and $q=1+\mu N/2$.
Now $1+R=\rho M/2$ and $1+Q=\sigma N/2$.
The first divisibility in \eqref{F2} holds since $N$ is even,
and the second holds since $M$ is even.
For \eqref{F3}, if either coefficient vanishes the product is zero;
otherwise the product is $MN/4$, divisible by $\max(M,N)$
because $\min(M,N)\ge8$. These are exactly the sixteen data in (I).

For positive $R$ and negative $Q$, the first equation of
\eqref{F2} reduces to $2N=0\pmod M$: the term $\rho NM/2$
already vanishes. Since $M,N$ are powers of two, this says $M\le2N$.
The other equation of \eqref{F2} is automatic.
The norm equations require $q=1$ modulo both $M/2$ and $N/2$,
or equivalently modulo $\max(M,N)/2$.
They leave $p$ arbitrary odd.

If $M\le N$, the two possible residues for $q$ are $1,1+N/2$.
The product in \eqref{F3} is
\[
 (2+\rho M/2)\sigma N/2
   =\sigma N+\rho\sigma MN/4,
\]
which is divisible by $N=\max(M,N)$ for either value of $\sigma$.
If $M=2N$, the condition on $q$ is modulo $N$, so $q=1$.
The same product is divisible by $M=2N$ exactly when $\sigma=0$:
the first term is $\sigma N$ and the second is divisible by $2N$.
This proves all assertions in (II).

For negative $R$ and positive $Q$, the second equation of
\eqref{F2} reduces to $2M=0\pmod N$, giving $N\le2M$.
The norm equations require $p=-1$ modulo both half-moduli and
leave $q$ arbitrary odd. If $N\le M$, both residues
$-1,-1+M/2$ are permitted, and the product
$\rho M(1+\sigma N/4)$ is divisible by $M$.
If $N=2M$, the product condition forces $\rho=0$,
and the condition on $p$ gives the single residue $-1$.
This proves (III).

Every possibility has now been checked against all five conditions.
The three sign patterns are distinct by Lemma~\ref{roots}, and within
each family the listed residues are distinct. Theorem~\ref{six}
therefore proves nonredundancy as marked triples.
\end{proof}

The following theorem deals with the boundary case $\min(M,N)<8$.
\begin{theorem}[Complete boundary lists]\label{boundary}
The cases with $\min(M,N)<8$ are as follows, together with their
versions obtained by exchanging the generators. All displayed choices
are independent unless a restriction is stated.
\begin{enumerate}[label=\textup{(\alph*)},leftmargin=2.5em]
\item If $M=N=2$, there is one datum:
$R=Q=p=q=1$.
\item If $N=2$ and $M\ge4$, there are four data:
\[
 Q=q=1,\qquad R=-1+\rho M/2,\qquad p=-1+\lambda M/2.
\]
\item If $M=N=4$, independently choose
$R,Q,p,q\in\{1,3\}$. There are sixteen data.
\item If $N=4$ and $M\ge8$, there are sixteen data:
\[
 R=-1+\rho M/2,\qquad p=-1+\lambda M/2,\qquad
 Q,q\in\{1,3\}.
\]
For $M=8$ there are additionally eight data:
\[
 R=1+4\rho,\qquad Q=3,\qquad q=1,\qquad
 p\in\{1,3,5,7\}.
\]
There are no additional data when $M\ge16$.
\end{enumerate}
Here $\rho,\lambda\in\{0,1\}$ and $P,S$ are always recovered by
\eqref{PS}. The exchange of generators acts on six-residue data by
\begin{equation}\label{swap}
 (P,Q,R,S,p,q)\longmapsto(S,R,Q,P,-q,-p),
\end{equation}
with $M,N$ exchanged.
\end{theorem}
\begin{proof}
We solve \eqref{F1}--\eqref{F5} in each boundary range.

If $M=N=2$, all four odd residues $R,Q,p,q$ have to be one.
Each of the five conditions then holds: the product in \eqref{F3}
is four, and both norm congruences are zero modulo two.
This proves (a).

If $N=2$, oddness gives $Q=q=1$.
Equation \eqref{F2} requires $2(1+R)=0\pmod M$, so
$R=-1$ modulo $M/2$. These are exactly the two values
$-1,-1+M/2$ for $M\ge4$.
With $q=1$, equation \eqref{F4} becomes
\[
 (R-1)(1+p)=0\pmod M.
\]
For $M\ge8$, each of the permitted $R-1$ is twice an odd integer,
so this means $p=-1\pmod{M/2}$, giving the two stated values.
For $M=4$, the permitted values of $R$ are one and three;
the displayed equation is automatic for either odd $p$.
These two odd values of $p$ are again exactly the stated pair.
The product equation \eqref{F3} is $2(1+R)=0\pmod M$,
already checked. Equation \eqref{F5} is automatic modulo two,
as is the second equation of \eqref{F2}. This proves (b).

If $M=N=4$, all odd residues square to one.
Each of $1+R,1+Q$ is even, so their product is zero modulo four;
the two equations of \eqref{F2} are automatic.
For \eqref{F4}, if $R=3$ both sides vanish modulo four.
If $R=1$, cancellation of a factor two asks
$1+pq=1+p\pmod2$, which holds since $q$ is odd.
For \eqref{F5}, $Q=3$ again makes both sides zero.
For $Q=1$, division by two asks
$1+pq=1-q\pmod2$, also true.
Thus all sixteen choices are admissible, proving (c).

Finally, let $N=4,M\ge8$.
For negative-sign $R=-1+\rho M/2$, equation \eqref{F4} forces
$p=-1\pmod{M/2}$. The two possible values of $Q$ and of $q$
are both allowed by \eqref{F1} and \eqref{F5}, by the same
modulo-four check just made.
The equation $4(1+R)=0\pmod M$ is automatic.
The product equation is also automatic: it is zero if $\rho=0$
or $Q=3$, and is $\rho M$ if $Q=1$.
The remaining divisibility equation holds because $M$ is divisible by
four. Thus these are the sixteen data in the first part of (d).

For positive-sign $R=1+\rho M/2$, the equation
$4(1+R)=0\pmod M$ reduces to $8=0\pmod M$, forcing $M=8$.
Here $1+R$ is twice an odd integer, so \eqref{F3} requires $Q=3$;
$Q=1$ would make the product have exactly two factors of two,
too few for divisibility by eight.
Equation \eqref{F4} requires $q=1\pmod4$, and hence $q=1$.
It leaves all four odd values of $p$ available.
The other equations hold, giving precisely the additional eight data.

For the exchange formula, the new actions are $\beta,\alpha$ in
the square basis $y,x$. Thus the action coordinates become
$(S,R,Q,P)$. The collection relation becomes $ab=t^{-1}ba$,
so its coordinates in that basis are $(-q,-p)$.
This proves \eqref{swap}, and it also proves that the exchanged lists
are exhaustive without redoing the calculations.
\end{proof}

\subsection{Enumeration of the abelian Frattini case}\label{enumeration}
The following theorem summarizes the abelian Frattini case.
\begin{theorem}\label{counts}
Let $\Bab(e,f)$ be the number of isomorphism classes of regular
dessins on the complete bipartite graph $K_{2^e,2^f}$ $(1\leq f\leq e)$ with nonmetacyclic
automorphism group $G$ and abelian Frattini subgroup. Then
\begin{equation}\label{countformula}
 \Bab(e,f)=
 \begin{cases}
 0,&f\le1,\\
 1,&f=e=2,\\
 4,&f=2<e,\\
 16,&f=e=3,\\
 24,&f=3,\ e=4,\\
 16,&f=3,\ e\ge5,\\
 16+2^{f+2},&f=e\ge4,\\
 16+3\cdot2^f,&f\ge4,\ e=f+1,\\
 16+2^{f+1},&f\ge4,\ e\ge f+2.
 \end{cases}
\end{equation}
\end{theorem}
\begin{proof}
If one exponent is one, the other cyclic factor has index two
in $G$ by exactness. A subgroup of index two is normal.
It is cyclic and has cyclic quotient of order two, so $G$ is
metacyclic. This proves the first line.

For $f=2,3$, Theorem~\ref{boundary} gives respectively one,
four, sixteen, twenty-four and sixteen data in the ranges specified.
No abstract-group identifications should be made in this count:
Theorem~\ref{six} shows that different data are different marked
triples, and Proposition~\ref{dessindictionary} identifies precisely
these with dessin classes.

For $M,N\ge8$, family (I) always contributes sixteen.
In family (II), if $M\le N$, there are two choices of $\rho$,
two of $\sigma$, $M/2$ of $p$, and two of $q$.
The product is $4M$.
If $M=2N$, $\sigma=0$ and $q=1$, leaving
$2(M/2)=M$ data. If $M>2N$, the family is empty.
Similarly, family (III) contributes $4N$ when $N\le M$,
$N$ when $N=2M$, and zero when $N>2M$.

We may assume $M\ge N$, since the generator exchange is a bijection.
If $M=N=2^{f-1}$, the total is
$16+4M+4N=16+2^{f+2}$.
If $M=2N$, the total is
$16+M+4N=16+6N=16+3\cdot2^f$.
If $M\ge4N$, only families (I) and (III) remain, giving
$16+4N=16+2^{f+1}$.
The disjointness proved in Theorem~\ref{largefamilies} completes
the count.
\end{proof}

\begin{remark}[Reciprocality and the diagonal]
The equality $\Bab(e,f)=\Bab(f,e)$ follows from \eqref{swap}.
For $e\ne f$, choosing which colour has the larger valency gives
one representative of each reciprocal pair; the corresponding
number is therefore $\Bab(e,f)$.
For $e=f$, the colour exchange acts on a single parameter set,
and its fixed points must be counted before taking orbits.
The colour-preserving count \eqref{countformula} does not divide
by two on the diagonal.
\end{remark}

For reference, the first values are:
\[
\begin{array}{c|rrrrrr}
 e\backslash f&2&3&4&5&6&7\\ \hline
2&1&4&4&4&4&4\\
3&4&16&24&16&16&16\\
4&4&24&80&64&48&48\\
5&4&16&64&144&112&80\\
6&4&16&48&112&272&208\\
7&4&16&48&80&208&528
\end{array}
\]
For example, the abelian Frattini count with one rotation order
eight is $24$ when the other order is $16$, and $16$ when the
other order is at least $32$.

\section{The nonabelian Frattini case}\label{nonabelian}
We now treat the complementary case $K=\Phi(G)$ nonabelian.
Use $e\ge f$ and $M=2^{e-1},N=2^{f-1}$; if the original
ordering is reversed, exchange it back after applying the lists.
Write $D=G'$ and retain $(G,a,b)$ for the distinguished triple.
Unlike the collection coefficient $t$ in the abelian presentation,
the symbol $c$ below is an actual commutator in its stated ordering.
The canonical scalar root is denoted $\xi$, to distinguish it
from the abelian-case face-square vector $z$.
When reproducing Janko's structural data, write $r_0,s_0$ for
his principal generators, $d$ for his second derived generator,
and $\zeta$ for a terminal involution. These auxiliary names do
not change the marking $(G,a,b)$.
\subsection{The parameters shared by all three families}
Use the least nonnegative representatives whenever an integer
polynomial is evaluated. Every family below requires
\begin{equation}\label{na:common}
 Q\equiv3\pmod4,\qquad Q^2\not\equiv1\pmod M,\qquad
 k\in\Z_N\text{ odd}.
\end{equation}
Define
\begin{equation}\label{na:poly}
 F_{Q,k}(Z)=\frac{Q-1}{2}Z^2+Z-\frac{\Nrm_{2k}(Q)}{2}.
\end{equation}
The integer $\Nrm_{2k}(Q)$ is divisible by four.
Lemma~\ref{na:root} proves that $F_{Q,k}$ has exactly one odd root
$\xi_0$ modulo $M/2$. Choose one of its two lifts
\begin{equation}\label{na:zR}
 \xi=\xi_0\ \text{or}\ \xi_0+M/2\pmod M,\qquad
 R=(1-Q)\xi-1\pmod M.
\end{equation}
Both lifts give the same $R$, and $R\equiv1\pmod4$.

\begin{theorem}[Nonabelian Frattini classification]\label{na:main}
All nonmetacyclic exact triples \eqref{2exact} with nonabelian
Frattini subgroup are given, without repetitions, by the following
three families. The common parameters are \eqref{na:common}--\eqref{na:zR}.
\begin{enumerate}[label=\textup{(\roman*)},leftmargin=2.5em]
\item Type $\mathcal C_0$ requires $Q^N=1\pmod M$ and has presentation
\begin{equation}\label{na:C0}
 \begin{split}
 G=\langle a,b,c\mid {}&
 b^{2N}=c^M=1,\quad [b,a]=c,\\
 &c^b=c^Q,\quad c^a=c^R,\quad
 a^2=b^{2k}c^\xi\rangle.
 \end{split}
\end{equation}
Here $D=\langle c\rangle\cong C_M$.
\item Type $\mathcal L_0$ requires $f\ge3$ and
$Q^{N/2}=1\pmod M$. Put $w=b^N$ and take
\begin{equation}\label{na:L0}
 \begin{split}
 G=\langle a,b,c\mid {}&
 b^{2N}=c^M=1,\quad [b,a]=c,\quad
 [w,a]=[w,c]=1,\\
 &c^b=wc^Q,\quad c^a=wc^R,\quad
 a^2=b^{2k}c^\xi\rangle .
 \end{split}
\end{equation}
The relation $[w,b]=1$ follows from $w=b^N$.
Here $D=\langle w\rangle\times\langle c\rangle\cong C_2\times C_M$.
\item Type $\mathcal C_1$ requires $e=f+1$, so $M=2N$, and has
presentation
\begin{equation}\label{na:C1}
 \begin{split}
 G=\langle a,b,c\mid {}&
 c^M=1,\quad a^M=c^{M/2},\quad [a,b]=c,\\
 &c^a=c^Q,\quad c^b=c^R,\quad
 b^2=a^{2k}c^\xi\rangle .
 \end{split}
\end{equation}
There is no further power restriction on $Q$.
Here $D=\langle c\rangle\cong C_M$.
\end{enumerate}
For $e=f$, every presentation in \textup{(i)} or \textup{(ii)}
also gives the distinct marked triple $(G,b,a)$.
These reciprocal triples are additional entries in the
colour-preserving classification.
There are no examples unless
\[
 e\ge5,\qquad f\ge3.
\]
For every such pair of exponents at least one example exists.
Each listed group has order $4MN$, its distinguished generators
have the orders in \eqref{2exact}, and their cyclic subgroups
intersect trivially.
\end{theorem}

The theorem is proved in Subsections~\ref{na:jankoreduction}--\ref{na:uniqueness}.
Only Janko's published structural classification is taken as an
external input; the exact-factor restrictions, canonical parameters,
extension construction and marked uniqueness are proved below.
Section~\ref{na:enumeration} then combines the two case counts.

\subsection{The reduction from Janko's presentations}\label{na:jankoreduction}
Janko's Theorem 5.1 \cite{Janko2008} says that a nonmetacyclic
bicyclic $2$-group has exactly one nonmetacyclic maximal subgroup.
The two cyclic factors lie in the other two maximal subgroups,
both metacyclic. His Section 4 exhausts these groups.

In the presentations relevant below, $D=G'$ is abelian,
\[
 G/D=\langle r_0D\rangle\times\langle s_0D\rangle,
 \qquad |r_0D|=2^\mu,\quad |s_0D|=2,\quad \mu\ge2,
\]
and $[r_0,s_0]=v$. The two metacyclic generating cosets are
\begin{equation}\label{na:genclasses}
 \mathcal F_0=\{r_0^i d:i\text{ odd},\ d\in D\},\qquad
 \mathcal F_1=\{r_0^i s_0d:i\text{ odd},\ d\in D\}.
\end{equation}
Every exact pair must contain one element from each.
Such a pair generates $G$ by the Frattini basis property.
Its cyclic subgroups are exact precisely when they are disjoint and
their orders have product $|G|$.

Theorems 4.1, 4.2, 4.6, 4.7, 4.11, 4.12 and 4.13(a) of \cite{Janko2008}
have abelian Frattini subgroup and are excluded from the present
nonabelian case. This is explicit in their statements or follows
immediately from their displayed actions: $\Phi(G)=\langle r_0^2,D\rangle$,
$D$ is abelian, and conjugation by $r_0^2$ is trivial on $D$.
The remaining candidates are Theorem 4.4, the two cases of Theorem 4.8,
Theorem 4.9 and Theorem 4.13(b).
The latter two can have either abelian or nonabelian Frattini subgroup,
but none of their cyclic factorisations is exact, as shown below.
Thus excluding them does not require an extra restriction on their
parameters.

\begin{proposition}[The surviving structural presentations]\label{na:survivors}
Every nonmetacyclic exact triple with nonabelian Frattini subgroup
comes from one of the following structural types.
\begin{enumerate}[label=\textup{(\alph*)},leftmargin=2.5em]
\item The cyclic-derived type of Janko's Theorem 4.4:
\begin{equation}\label{na:Jcyclic}
 \begin{split}
 D=\langle v\rangle,\quad &|v|=2^\nu,\quad
 \zeta=v^{2^{\nu-1}},\\
 r_0^{2^\mu}=\zeta^\delta,\quad&
 s_0^2=\zeta^\eta,\quad [r_0,s_0]=v,\quad
 v^{r_0}=v^q,\quad v^{s_0}=v^{-1},
 \end{split}
\end{equation}
where $\nu,\mu\ge2$, $\delta,\eta\in\{0,1\}$,
$q\equiv3\pmod4$ and $q^{2^{\mu-1}}=1\pmod{2^\nu}$.
Exactness requires either
\[
 \delta=0,\quad \nu\ge\mu-1,
 \qquad\text{or}\qquad
 \delta=1,\quad \nu=\mu.
\]
In the first range $(e,f)=(\nu+1,\mu)$; in the second
$(e,f)=(\mu+1,\mu)$.
\item The first case of Janko's Theorem 4.8:
\begin{equation}\label{na:Jranktwo}
 \begin{split}
 D=\langle u\rangle\times\langle v\rangle,\quad&
 |u|=2,\quad |v|=2^\nu,\quad \zeta=v^{2^{\nu-1}},\\
 r_0^{2^\mu}=u\zeta^\eta,\quad&
 s_0^2=\zeta^\kappa,\quad [r_0,s_0]=v,\\
 u^{r_0}=u^{s_0}=u,\quad&
 v^{r_0}=uv^q,\quad v^{s_0}=v^{-1},
 \end{split}
\end{equation}
where $\nu,\mu\ge2$, $\eta,\kappa\in\{0,1\}$,
$q\equiv3\pmod4$ and $q^{2^{\mu-1}}=1\pmod{2^\nu}$.
Exactness requires $\nu\ge\mu$, and then
$(e,f)=(\nu+1,\mu+1)$.
\end{enumerate}
In both types the Frattini subgroup is nonabelian exactly when
$q^2\ne1\pmod{2^\nu}$.
\end{proposition}

\begin{proof}
We use Janko's classification for the two structural presentations. We verify the additional restrictions and
the exclusion of the remaining candidates.

In either type $\Phi(G)=\langle r_0^2,D\rangle$.
For \eqref{na:Jcyclic}, conjugation by $r_0^2$ on $v$ is $v\mapsto v^{q^2}$.
For \eqref{na:Jranktwo}, applying the action twice gives
$(uv^q)^{r_0}=u^{1+q}v^{q^2}=v^{q^2}$.
It fixes $u$. Since $D$ is abelian, the final assertion follows.

\emph{Orders in the cyclic-derived type.}
The power relation and the norm lemma show that every element of
$\mathcal F_0$ has order $2^{\mu+\delta}$.
Indeed, for odd $i$ its $2^\mu$-th power is $\zeta^\delta$;
the norm of the $v$-component vanishes because
$q^{2^{\mu-1}}=1$.
For $y\in\mathcal F_1$, the corresponding calculation gives
\begin{equation}\label{na:JpowerC}
 y^{2^\mu}=\zeta^\delta v^{2^{\mu-1}t},\qquad t\text{ odd}.
\end{equation}
If $\nu<\mu-1$, the second factor is one.

Here is why the coefficient in \eqref{na:JpowerC} is odd.
For odd $i$, the commutator $[r_0^i,s_0]$ is a geometric sum in $v$
with odd coefficient. Its contribution to the $2^\mu$-th power uses
a norm of length $2^{\mu-1}$ for $q^{2i}\equiv1\pmod8$.
By \eqref{na:positiveval} that contribution has valuation exactly
$\mu-1$. The freely chosen $v$-component of $y$ uses the norm for
$-q^i\equiv1\pmod4$ of length $2^\mu$, whose valuation is $\mu$.
It cannot change the oddness of the previous coefficient.
The $s_0^2$ contribution has disappeared since $\mu\ge2$.

If $\delta=0$ and $\nu\ge\mu$, the two orders are
$2^\mu,2^{\nu+1}$ and their terminal involutions lie respectively
outside $D$ and at $\zeta$.
If $\delta=0$ and $\nu=\mu-1$, both orders are $2^\mu$,
with terminal involutions $r_0^{2^{\mu-1}}$ and
$r_0^{2^{\mu-1}}\zeta$. The norms at this half-power do not change
them, by the same valuation calculation.
In either range the two cyclic subgroups are disjoint and their
orders multiply to $|G|=2^{\nu+\mu+1}$.
If $\nu<\mu-1$, both orders are $2^\mu$, and their product is too
large to be an exact factorisation.

If $\delta=1$ and $\nu=\mu$, the factors in \eqref{na:JpowerC}
cancel, giving orders $2^{\mu+1}$ in $\mathcal F_0$ and
$2^\mu$ in $\mathcal F_1$.
The first terminal involution is $\zeta$, while the second has
nontrivial image $r_0^{2^{\mu-1}}D$; the pair is exact.
If $\nu>\mu$, the orders are $2^{\mu+1}$ and $2^{\nu+1}$,
with product twice $|G|$.
If $\nu<\mu$, both orders are $2^{\mu+1}$ and their product
again exceeds $|G|$. These are exactly the restrictions in (a).

\emph{Orders in the rank-two type.}
All elements of $\mathcal F_0$ have order $2^{\mu+1}$,
with terminal involution $u\zeta^\eta$.
The same norm calculation gives, for $y\in\mathcal F_1$,
\begin{equation}\label{na:JpowerL}
 y^{2^\mu}=u\zeta^\eta v^{2^{\mu-1}t},\qquad t\text{ odd}.
\end{equation}
The additional $u$ factors in the norms occur an even number of times,
since $\mu\ge2$.
If $\nu=\mu$, both coset orders are $2^{\mu+1}$, with different
terminal involutions $u\zeta^\eta,u\zeta^{\eta+1}$.
If $\nu>\mu$, the second order is $2^{\nu+1}$ and its terminal
involution is $\zeta$, distinct from the first.
In these ranges the order product is $2^{\nu+\mu+2}=|G|$.
If $\nu<\mu$, both coset orders are $2^{\mu+1}$ and the product
is too large. This proves the exactness restriction in (b).
The extra multiplier restriction follows also from
$[r_0^{2^\mu},s_0]=1$: its norm is zero, and
\eqref{na:normval} makes this equivalent to
$q^{2^{\mu-1}}=1$.

\emph{The second case of Theorem 4.8.}
Here $r_0^{2^\mu}=uv^{2^{\nu-2}}\zeta^\eta$,
$u^{r_0}=u\zeta$, $v^{r_0}=uv^q$, and $\nu\ge\mu+2$.
The length-$2^\mu$ norm of $v$ is $\zeta$, because
$[r_0^{2^\mu},s_0]=\zeta$.
Thus an element of $\mathcal F_0$ has $2^\mu$-th power
$uv^{2^{\nu-2}t}\zeta^j$ with $t$ odd.
This power has order four and terminal involution $\zeta$.
For $\mathcal F_1$, the norm of the commutator has
$v$-valuation $\mu-1$, whereas the norm of an arbitrary
kernel component has valuation at least $\mu$.
Since $\nu-2\ge\mu$, the term from $r_0^{2^\mu}$ cannot cancel
the former. Hence an element of $\mathcal F_1$ has order
$2^{\nu+1}$ and terminal involution $\zeta$ too.
All proposed factors have a common involution, so none is exact.

\emph{Theorem 4.9.}
Rename its derived generators $v,d$, both of order $2^h$, $h\ge2$.
Its actions include
\[
 v^{r_0}=vd,\quad d^{r_0}=d^{-1}\zeta^\eta,\quad
 v^{s_0}=v^{-1}\zeta^\epsilon,\quad d^{s_0}=d^{-1},
\]
and $r_0^4=g\omega$, where
$g=v^{-2}d^{-1}$ and $\omega\in\Omega_1(D)$.
The four-term norm for conjugation by an odd power of $r_0$
takes $v$ to $g^{-2}$ and $d$ to one.
Therefore fourth powers in $\mathcal F_0$ have the form
$g^{\,\text{odd}}\omega'$.
For an odd power of $r_0$ followed by $s_0$, the four-term norms
take $v$ to $d^{-2}$ and $d$ to $d^4$.
Combining these with the square commutator formula gives fourth
powers $d^{\,\text{odd}}\omega''$ in $\mathcal F_1$.
The factors $\omega',\omega''$ have order at most two.
Both $g$ and $d$ have order $2^h$, so both coset element orders
are $2^{h+2}$. Their product is $2^{2h+4}$, whereas
$|G|=2^{2h+3}$. Thus no exact pair exists.

\emph{Theorem 4.13(b).}
The same four-term norm calculation applies, with
\[
 d^{2^h}=v^{-2^{h+1}}\zeta,\qquad
 |v|=2^{h+t+1},\quad |d|=2^{h+t},\quad h,t\ge2.
\]
For $g=v^{-2}d^{-1}$ this relation gives $g^{2^h}=\zeta$.
Fourth powers in $\mathcal F_0$ are odd powers of $g$ times
involutions, and therefore have terminal involution $\zeta$.
Fourth powers in $\mathcal F_1$ are odd powers of $d$ times
involutions. Since
$d^{2^{h+t-1}}=\zeta$, their terminal involution is again $\zeta$.
Multiplying an odd power of an element of order at least four by
a commuting involution does not change its terminal involution:
raising to half its order kills the extra involution.
Consequently every candidate pair intersects nontrivially.

These exclusions, together with the abelian Frattini cases already
discarded, exhaust Janko's list. This proves the proposition.
\end{proof}
\subsection{Canonical parameters and the quadratic congruence}\label{na:canonical}
For the proofs it is convenient to use the \emph{preferred extension
ordering} $(r,s)$:
\begin{equation}\label{na:preferred}
 (r,s)=(b,a)\quad\text{in }\mathcal C_0,\mathcal L_0,\qquad
 (r,s)=(a,b)\quad\text{in }\mathcal C_1 .
\end{equation}
Thus the marked triple remains $(G,a,b)$ throughout.
The symbols $r,s$ are auxiliary names for its two generators,
not a different convention for triples.
For the additional diagonal reciprocal entry, exchange the marking.

Put $c=[r,s]$, so $r^s=rc$.
Write $\varepsilon=0$ in the cyclic-derived types and
$\varepsilon=1$ in the rank-two type; when $\varepsilon=0$ all factors
$w^\varepsilon$ are omitted.
The canonical relations become
\begin{equation}\label{na:unified}
 c^r=w^\varepsilon c^Q,\qquad
 c^s=w^\varepsilon c^R,\qquad s^2=r^{2k}c^\xi,
\end{equation}
where $w=r^N$ is central of order two in $\mathcal L_0$.

\begin{proposition}[Necessity of the canonical form]\label{na:necessity}
Every exact triple in Proposition~\ref{na:survivors} with nonabelian
Frattini subgroup gives precisely the parameter restrictions in
Theorem~\ref{na:main}, together with
\begin{equation}\label{na:compat}
 R=(1-Q)\xi-1,\qquad
 \Nrm_{2k}(Q)+(R-1)\xi=0\pmod M.
\end{equation}
\end{proposition}
\begin{proof}
Choose $r\in\mathcal F_0$ and $s\in\mathcal F_1$.
In the cyclic-derived type, $c=[r,s]$ is an odd power of $v$:
modulo $\langle v^2\rangle$ the derived generator is central,
and the two odd generating cosets have commutator $v$.
Thus $|c|=2^\nu=M$ and $D=\langle c\rangle$.
In the rank-two type the same argument modulo
$\langle u,v^2\rangle$ shows that the $v$-coordinate of $c$ is odd,
so again $|c|=M$.

For the cyclic-derived case with $\delta=0$, the orders from
Proposition~\ref{na:survivors} give $|r|=2N$ and
$\langle r\rangle\cap D=1$.
The subgroup $H=\langle r,c\rangle$ has unique normal forms
$r^i c^j$, $i\in\Z_{2N}$, $j\in\Z_M$.
Conjugation by $r$ on $D$ is an odd power of the original
$q$ action, so $c^r=c^Q$ with $Q\equiv3\pmod4$ and $Q^N=1$.
Conjugation by $s$ is that odd action followed by inversion,
so $c^s=c^R$ with $R\equiv1\pmod4$.

For the rank-two case, the terminal involution of $r$ is
\[
 w=r^N=u\zeta^\eta .
\]
It is central, and it lies outside $\langle c\rangle$:
the only involution in $\langle c\rangle$ is $\zeta$,
whereas $w$ has a nonzero $u$-coordinate.
Hence $D=\langle w\rangle\times\langle c\rangle$.
The action of any odd power of $r_0$ has a nonzero $u$ factor
on an element with odd $v$-coordinate. Replacing $u$ by $w$
therefore gives
$c^r=wc^Q$ and $c^s=wc^R$,
with $Q\equiv3$ and $R\equiv1$ modulo four.
Changing this basis changes the original multiplier only by
a multiple of $M/2$, so it preserves these congruences.
It also preserves the restriction $Q^{N/2}=1$, since $N/2$ is even.
The subgroup $H=\langle r,c\rangle$ again has unique normal forms
$r^i c^j$ with periods $2N,M$; the $w$ factor has been absorbed
by the relation $w=r^N$.

For the cyclic-derived case with $\delta=1$, exactness gives
$e=f+1$. Now $r$ is the longer generator, $|r|=2M$,
$r^M=c^{M/2}$, and $H$ has unique normal forms
$r^i c^j$ with $i,j\in\Z_M$.
Its two conjugation multipliers still satisfy $Q\equiv3$,
$R\equiv1$ modulo four.
For every such $Q$, Lemma~\ref{na:normlemma} gives
$Q^{M/4}=1$ when $M\ge8$, so no further multiplier restriction
is needed.

In every case $s^2$ lies in $H$ and has a unique normal form
$r^{2k}c^\xi$, with $k$ taken modulo $N$.
The images of $r,s$ in the cyclic factor of $G/D$ have odd
exponents relative to each other. Thus $k$ is odd.
In the rank-two case the abelianization first fixes $k$ modulo
$N/2$, but both lifts modulo $N$ are odd.

Conjugate $r$ by $s^2$ in two ways.
On one hand, using $r^s=rc$,
\[
 r^{s^2}=(rc)^s=r\,w^\varepsilon c^{1+R}.
\]
On the other hand, $r$ commutes with its own powers and
\[
 r^{r^{2k}c^\xi}=r^{c^\xi}
    =r\,w^{\varepsilon \xi}c^{(1-Q)\xi}.
\]
Thus $1+R=(1-Q)\xi\pmod M$.
Since both $1+R$ and $1-Q$ are two modulo four, $\xi$ must be odd.
This also makes the $w$ factors agree.

Finally $s$ fixes its square. The right-action power formula gives
\[
 (rc)^{2k}=r^{2k}w^{\varepsilon k}c^{\Nrm_{2k}(Q)}.
\]
Therefore
\[
 (r^{2k}c^\xi)^s
   =r^{2k}w^{\varepsilon(k+\xi)}
       c^{\Nrm_{2k}(Q)+Rz}.
\]
Because $k,\xi$ are odd, the $w$ factor is one.
Comparison with $s^2=r^{2k}c^\xi$ gives the second congruence of
\eqref{na:compat}.

The action of $r^2$ on $c$ is $c\mapsto c^{Q^2}$.
It is the only possible nontrivial action in
$\Phi(G)=\langle r^2,D\rangle$, so nonabelianness is precisely
$Q^2\ne1\pmod M$.
Since all units congruent to three modulo four square to one
modulo $2,4,8$, this requires $M\ge16$, or $e\ge5$.
If $f=2$, the cyclic-derived restriction $Q^N=Q^2=1$
rules out $\mathcal C_0$; the other two structural cases also
cannot occur. This proves the claimed empty ranges.
\end{proof}

\begin{lemma}[The unique odd root]\label{na:root}
For $Q\equiv3\pmod4$ and odd $k$, the polynomial \eqref{na:poly}
has exactly one odd root modulo $M/2$.
Its two lifts modulo $M$ are exactly the odd solutions of
\eqref{na:compat}, with $R$ defined by \eqref{na:zR}.
Moreover
\begin{equation}\label{na:Rsquare}
 R^2=Q^{2k}\pmod M.
\end{equation}
The root is obtained without a search: start with $\xi_1=1$ and put
\begin{equation}\label{na:lift}
 \xi_{j+1}=\xi_j+2^j
       \left(\frac{F_{Q,k}(\xi_j)}{2^j}\bmod2\right)
       \qquad(1\le j<e-2).
\end{equation}
Then $\xi_0=\xi_{e-2}$, taken modulo $M/2$.
\end{lemma}
\begin{proof}
The factorisation
$\Nrm_{2k}(Q)=(1+Q)\Nrm_k(Q^2)$ shows divisibility by four.
The second factor is odd because it is a sum of an odd number
of odd terms.
Consequently $(Q-1)/2$ is odd and the constant term of
$F_{Q,k}$ is even.
Modulo two the polynomial is $Z^2+Z$.
It has one root of odd parity, namely one.
There is also an even root; it is irrelevant because $\xi$ must
be odd.

The derivative
$F'_{Q,k}(Z)=(Q-1)Z+1$ is odd for every integer $Z$.
Suppose $\xi_j$ is an odd root modulo $2^j$.
The two candidate lifts are $\xi_j$ and $\xi_j+2^j$.
Taylor expansion modulo $2^{j+1}$ gives
\[
 F(\xi_j+2^j\epsilon)=F(\xi_j)+2^j\epsilon F'(\xi_j)
                  \pmod{2^{j+1}}.
\]
Exactly one $\epsilon\in\{0,1\}$ makes this zero,
and it is the residue of $F(\xi_j)/2^j$ modulo two.
This proves both uniqueness and the recursion.

Substituting $R=(1-Q)\xi-1$ into the second congruence of
\eqref{na:compat} yields
\[
 \Nrm_{2k}(Q)-(Q-1)\xi^2-2\xi=0\pmod M.
\]
Division by two is equivalent to $F(\xi)=0\pmod{M/2}$,
not to a congruence modulo $M$.
Thus the two lifts of the unique odd root give exactly all solutions.
Their values of $R$ agree since $(1-Q)M/2$ is divisible by $M$.

To prove \eqref{na:Rsquare}, multiply the second equation of
\eqref{na:compat} by $1-Q$.
The geometric-sum identity gives
$(1-Q)\Nrm_{2k}(Q)=1-Q^{2k}$,
and $(1-Q)\xi=R+1$.
The result is
$1-Q^{2k}+(R-1)(R+1)=0\pmod M$,
which is the desired identity.
\end{proof}
\subsection{Existence, exactness and the Frattini subgroup}\label{na:existence}
We construct each presented group explicitly. This is needed to
ensure that the presentations do not collapse the asserted kernel
or the orders of the two distinguished elements.

\begin{lemma}[A cyclic extension with right conjugation]\label{na:cyclicextension}
Let $H$ be a finite group, $\sigma\in\Aut(H)$ a right automorphism,
$t\in H$ and $n\ge1$. Suppose
\[
 \sigma^n=\iota_t,\qquad t^\sigma=t .
\]
Then the relative presentation
\[
 \langle H,s\mid h^s=h^\sigma\ (h\in H),\ s^n=t\rangle
\]
contains $H$ faithfully, has order $n|H|$, and has unique normal
forms $s^i h$, $0\le i<n$, $h\in H$.
\end{lemma}
\begin{proof}
First form the extension of $H$ by an infinite cyclic generator $s$
with $h^s=h^\sigma$. Its elements have unique forms $s^i h$,
$i\in\Z$. The element $s^nt^{-1}$ centralizes $H$, because its
conjugation is $\sigma^n\iota_{t^{-1}}=\id$.
It commutes with $s$ because $t^\sigma=t$.
It is therefore central.

The infinite cyclic subgroup generated by $s^nt^{-1}$ meets $H$
trivially: its nonzero powers have nonzero image in the infinite
cyclic quotient. Quotienting by this subgroup thus retains a
faithful copy of $H$ and imposes exactly $s^n=t$.
Its image in the infinite cyclic quotient is $n\Z$, so the new
quotient over $H$ has order $n$. Reducing the exponent of $s$
gives the stated normal forms, which are unique by these facts.
\end{proof}

Use the preferred generators in \eqref{na:preferred}.
For $\mathcal C_0$, start with $\langle c\rangle\cong C_M$,
let $c^\rho=c^Q$, and adjoin $r$ with $r^{2N}=1$.
Since $Q^N=1$, the action has $2N$-th power one.
Lemma~\ref{na:cyclicextension} gives $H=\langle r,c\rangle$
of order $2NM$, with unique forms $r^i c^j$,
$i\in\Z_{2N}$, $j\in\Z_M$.

For $\mathcal L_0$, start instead with the independent product
\[
 E=\langle w\rangle\times\langle c\rangle\cong C_2\times C_M,
\]
and the right automorphism $w^\rho=w$, $c^\rho=wc^Q$.
Since $N$ is even and $Q^{N/2}=1$, we have $\rho^N=\id$.
Adjoin $r$ with $r^N=w$.
The extension lemma gives $|H|=N|E|=2NM$.
Absorbing $w$ into $r^N$ gives the same unique forms
$r^i c^j$, with periods $2N,M$.
This also proves that $w$ remains a nontrivial involution
independent of $\langle c\rangle$.

For $\mathcal C_1$, start with $\langle c\rangle\cong C_M$,
let $c^\rho=c^Q$, and adjoin $r$ with $r^M=c^{M/2}$.
For an odd unit modulo $M$, its $M$-th power is one and it fixes
$c^{M/2}$. The extension lemma gives $|H|=M^2$ and unique
forms $r^i c^j$ with periods $M,M$.
The relation $r^M=c^{M/2}$ shows that $r$ has order $2M$.

Define a right map on $H$ by
\begin{equation}\label{na:sigma}
 r^\sigma=rc,\qquad c^\sigma=w^\varepsilon c^R,
 \qquad w^\sigma=w\quad\text{when }\varepsilon=1.
\end{equation}
We check that it is an automorphism, then apply the extension lemma
with $s^2=t_0=r^{2k}c^\xi$.

In $\mathcal C_0$, the power relation is preserved because
\[
 (rc)^{2N}=r^{2N}c^{\Nrm_{2N}(Q)}=1 .
\]
Indeed, $Q^N=1$ and Lemma~\ref{na:normlemma} imply
$\Nrm_{2N}(Q)=0\pmod M$.
For $\mathcal L_0$, $Q^{N/2}=1$ similarly gives
$\Nrm_N(Q)=0\pmod M$.
The norm also includes a factor $w^{N/2}=1$, since $N/2$ is even.
Therefore $(rc)^N=w$, as required.
For $\mathcal C_1$, the norm $\Nrm_M(Q)$ vanishes modulo $M$,
so $(rc)^M=c^{M/2}=(c^R)^{M/2}$.
This preserves the defining nonsplit power relation.

The conjugation relation is also preserved.
In the cyclic-derived cases both images of $c^r=c^Q$
give $c^{QR}$.
In the rank-two case, applying $\sigma$ to $c^r=wc^Q$ gives
\[
 (wc^Q)^\sigma=w^{1+Q}c^{RQ}=c^{RQ}.
\]
Conjugating $c^\sigma=wc^R$ by $r^\sigma=rc$ gives
$w^{1+R}c^{QR}=c^{QR}$ too.
All centrality relations on $w$ are preserved.
Thus \eqref{na:sigma} defines an endomorphism.
It is invertible on the normal abelian subgroup used to construct
$H$, since $R$ is odd, and it induces the identity on the cyclic
quotient of $H$ by that subgroup. It is therefore surjective
and, since $H$ is finite, an automorphism.

\begin{proposition}[The second extension is consistent]\label{na:sigmaidentity}
The automorphism \eqref{na:sigma} satisfies
\[
 t_0^\sigma=t_0,\qquad \sigma^2=\iota_{t_0},
 \qquad t_0=r^{2k}c^\xi.
\]
\end{proposition}
\begin{proof}
The power calculation used in Proposition~\ref{na:necessity} gives
\[
 t_0^\sigma
   =r^{2k}w^{\varepsilon(k+\xi)}
       c^{\Nrm_{2k}(Q)+Rz}=r^{2k}c^\xi,
\]
because $k,\xi$ are odd and \eqref{na:compat} holds.

On $r$, the second iterate is
$r^{\sigma^2}=r w^\varepsilon c^{1+R}$.
Conjugation by $t_0$ gives
$r^{t_0}=r w^{\varepsilon \xi}c^{(1-Q)\xi}$.
These agree by \eqref{na:compat} and oddness of $\xi$.
On $c$, the second iterate is $c^{R^2}$:
when $\varepsilon=1$ the two $w$ factors cancel because $R$ is odd.
Conjugation by $t_0$ gives $c^{Q^{2k}}$,
since $r^{2k}$ has no $w$ factor in its action.
The equality of these two powers is \eqref{na:Rsquare}.
Both maps fix $w$ where it occurs.
They agree on generators, hence are equal automorphisms.
\end{proof}

Apply Lemma~\ref{na:cyclicextension} with $n=2$.
This constructs $G$ of order $2|H|=4MN$ in all three types:
in $\mathcal C_1$, $2M^2=4MN$ because $M=2N$.
Its relative relations are exactly those in
\eqref{na:C0}, \eqref{na:L0} or \eqref{na:C1}.
Conversely those relations collect every word into an $H$ normal
form and one binary $s$-coordinate, so there are at most the number
of elements just constructed. The presentation therefore defines
exactly this group.

\begin{proposition}[Orders and exactness]\label{na:exactness}
The constructed groups have the generator orders asserted in
Theorem~\ref{na:main}, and the two cyclic factors are disjoint.
\end{proposition}
\begin{proof}
In $\mathcal C_0,\mathcal L_0$, $r$ has order $2N$.
The action of $r^{2k}$ on $c$ is multiplication by $Q^{2k}$,
which is one modulo eight.
Using $s^2=r^{2k}c^\xi$, the norm lemma gives
\[
 s^M=r^{kM}c^{M/2}.
\]
This is an equality of elements, not just of orders:
the relevant norm has valuation $e-2$, so modulo $M$
its product with odd $\xi$ is exactly $M/2$.
If $M>N$, then $2N$ divides $M$ and the $r$ factor is one.
Thus $s^M=c^{M/2}\ne1$ and $|s|=2M$.
The terminal involution $c^{M/2}$ differs from $r^N$,
because $\langle r\rangle\cap\langle c\rangle=1$.
If $M=N$, instead $s^M=r^N c^{M/2}$;
this is nontrivial and differs from $r^N$ by independence.
Again $|s|=2M$ and the two terminal involutions are different.

In $\mathcal C_1$, $r$ has order $2M$ and terminal involution
$c^{M/2}$.
The same norm calculation now gives
\[
 s^M=r^{kM}c^{M/2}
     =c^{(M/2)(k+1)}=1,
\]
since $k$ is odd.
On the other hand, the image of
$s^{M/2}=(r^{2k}c^\xi)^{M/4}$
in $H/\langle c\rangle$ is
$r^{kM/2}\langle c\rangle$, which is nontrivial.
Hence $|s|=M=2N$, and its terminal involution lies outside
$\langle c\rangle$. It differs from that of $r$.

The terminal-involution criterion proves disjointness in every case.
The product of the generator orders is $4MN=|G|$.
Their cyclic subgroups therefore have setwise product all of $G$,
which proves exactness in the specified marking.
\end{proof}

\begin{proposition}[Derived and Frattini subgroups]\label{na:frattini}
For the three constructed types the derived subgroup is the one
stated in Theorem~\ref{na:main}. In preferred ordering put $A=r^2$.
Then
\[
 K=\Phi(G)=\langle A,c\rangle,
\]
and its presentations are
\begin{align}
 \mathcal C_0,\mathcal L_0:\quad
 K&=\langle A,c\mid A^N=c^M=1,\ c^A=c^{Q^2}\rangle,\label{na:Ksplit}\\
 \mathcal C_1:\quad
 K&=\langle A,c\mid c^M=1,\ A^N=c^{M/2},\
                         c^A=c^{Q^2}\rangle.\label{na:Knonsplit}
\end{align}
They are nonabelian metacyclic groups of order $MN$.
The full group $G$ is nonmetacyclic.
\end{proposition}
\begin{proof}
Since $c=[r,s]$, its normal closure is $G'$:
quotienting by that closure makes the two generators commute.
In the cyclic-derived cases it is already the normal subgroup
$\langle c\rangle$.
In $\mathcal L_0$, the relation $c^r=wc^Q$ shows that this closure
also contains $w$. The subgroup $\langle w,c\rangle$ is normal
and abelian and has the independent factors constructed above.
This proves the derived subgroup assertions.

The subgroup $\langle A,c\rangle$ contains $D$,
because $w=A^{N/2}$ in the rank-two case.
It contains both generating squares and has index four,
with elementary abelian quotient on $rK,sK$.
Consequently $G^2G'\le K$, while $A=r^2$ and $c\in G'$
give the reverse inclusion into $\Phi(G)=G^2G'$.
This proves equality.
The $H$ normal forms restrict to unique forms $A^i c^j$,
$i\in\Z_N,j\in\Z_M$, giving \eqref{na:Ksplit},\eqref{na:Knonsplit}.
The action is $c^A=c^{Q^2}$, since the two $w$ factors cancel.
Thus $K$ is nonabelian exactly as required by \eqref{na:common};
it is metacyclic with cyclic normal subgroup $\langle c\rangle$.

To show that $G$ is nonmetacyclic, put
$K_1=\langle A^2,c^2\rangle$.
The relation $c^A=c^{Q^2}$ gives
$K'=\langle c^{Q^2-1}\rangle\le\langle c^2\rangle\le K_1$.
Thus $K_1$ is normal in $K$.
The normal forms show that
$K/K_1\cong C_2^2$ on $A,c$.
Since $K_1$ is generated by squares and this quotient is elementary
abelian, $K_1=K^2K'=\Phi(K)$.
It is therefore characteristic in $K$ and normal in $G$.
Both $r,s$ act trivially
on $K/K_1$.
For $r$ this follows from $c^r=w^\varepsilon c^Q$,
with $Q$ odd and $w=A^{N/2}\in K_1$.
For $s$ use
\[
 A^s=(rc)^2=A w^\varepsilon c^{Q+1},\qquad
 c^s=w^\varepsilon c^R;
\]
the added factors lie in $K_1$.
Here $N/2$ is even whenever $w$ occurs.

A direct collection gives
\begin{equation}\label{na:facesquare}
 (rs)^2=A^{k+1}w^\varepsilon c^{\xi-R}.
\end{equation}
The exponents $k+1$ and $\xi-R$ are even, so this square lies in $K_1$.
Therefore
$\langle rs,K\rangle/K_1\cong C_2^3$:
two independent generators come from $K/K_1$ and the third is
$rsK$. This subgroup needs at least three generators.

Every subgroup of a metacyclic group is metacyclic:
intersect it with a cyclic normal subgroup and take its image
in the cyclic quotient. Such a subgroup is generated by at most
two elements, a generator of the intersection and a lift of a
generator of the image.
The three-generator subgroup just found is therefore impossible
inside a metacyclic $G$. Hence $G$ is nonmetacyclic.
\end{proof}

\begin{remark}[The shorter terminal involution need not be central]
In $\mathcal C_0$ no centrality relation on $b^N$ is imposed.
Indeed,
\[
 (b^N)^a=(bc)^N=b^N c^{\Nrm_N(Q)}.
\]
The restriction $Q^N=1$ ensures that the longer norm
$\Nrm_{2N}(Q)$ vanishes modulo $M$; it does not always make
$\Nrm_N(Q)$ vanish.
For $M=16,N=4,Q=3$, the latter norm is $40=8\pmod{16}$,
so $(b^4)^a=b^4c^8\ne b^4$.
Inserting centrality of $b^N$ would therefore delete valid entries.
\end{remark}
\subsection{Marked uniqueness and reciprocality}\label{na:uniqueness}
\begin{proposition}[Intrinsic recovery of every parameter]\label{na:unique}
The lists in Theorem~\ref{na:main} are nonredundant as exact triples.
No nonabelian Frattini triple on the diagonal is symmetric.
\end{proposition}
\begin{proof}
First, $|D|=2M$ distinguishes $\mathcal L_0$ from the two
cyclic-derived types, where $|D|=M$.

For a cyclic derived group, the multiplier of a distinguished
generator on $D$ modulo four is intrinsic: replacing a generator
of $D$ by an odd power does not change that multiplier.
In $\mathcal C_0$, the longer marked generator $a$ has multiplier
$R\equiv1\pmod4$, whereas in $\mathcal C_1$ it has multiplier
$Q\equiv3\pmod4$.
This separates the types when both occur at $e=f+1$.

For $\mathcal L_0$, use the characteristic subgroup
$\Omega_1(D)$, the subgroup of elements of order at most two.
The quotient $D/\Omega_1(D)$ is cyclic of order $M/2$.
The $w$ factors disappear in this quotient, so the two preferred
generators have multipliers $Q$ and $R$, respectively.
Their residues modulo four are intrinsic as well, since $M\ge16$.
On the diagonal, in either type the two distinguished generators
therefore induce different multipliers modulo four, one three and
one one. Exchanging the marking changes this invariant.
The two diagonal orientations are consequently distinct.
This also proves that no automorphism interchanges the two
distinguished generators, so none of these dessins is symmetric.

Once the type and marking are fixed, the preferred ordering $(r,s)$
is determined. Then $c=[r,s]$ is determined by the triple.
In the rank-two case $w=r^N$ is determined too.
The relation $c^r=w^\varepsilon c^Q$ gives the unique residue $Q$.
The unique normal form of $s^2$ in $H$ gives $2k,\xi$,
hence $k\in\Z_N$ and $\xi\in\Z_M$.
The period of the $r$ coordinate is $2N$ in the first two types;
it is $M=2N$ in the third, so the same conclusion holds there.
Thus an isomorphism of triples preserves every canonical parameter.
In particular the two lifts of $\xi_0$ are different marked triples,
although they yield the same $R$.

Conversely, equal type, marking and parameters give the identical
presentation with the identical distinguished pair.
This proves both uniqueness and the stated symmetry assertion.
\end{proof}

\begin{proof}[Completion of Theorem~\ref{na:main}]
Janko's theorem and Proposition~\ref{na:survivors} exhaust the relevant
structural groups and their possible generating cosets.
Proposition~\ref{na:necessity} converts every exact pair to a displayed
canonical type.
Lemma~\ref{na:root} solves its remaining congruences.
The construction and Propositions~\ref{na:exactness},\ref{na:frattini}
prove that every admissible parameter really gives an exact,
nonmetacyclic group with nonabelian Frattini subgroup.
Proposition~\ref{na:unique} removes all marked repetitions.

Finally, examples exist in every range $e\ge5,f\ge3$:
in $\mathcal C_0$ take $Q=-1+M/4$.
This residue is three modulo four, has $Q^2\ne1$ modulo $M$,
and has order four, hence $Q^N=1$ since $N\ge4$.
Choose any odd $k$ and either root lift.
The empty ranges were established in Proposition~\ref{na:necessity}.
\end{proof}

\subsection{The nilpotency class of the Frattini subgroup}\label{na:invariants}
\begin{proposition}
For an admissible datum put
\[
 v=v_2(Q^2-1),\qquad h=e-1.
\]
Then
\[
 \cl(K)=\left\lceil\frac{h}{v}\right\rceil .
\]
In particular, nonabelian Frattini subgroups of class greater than
two occur in the classification.
\end{proposition}
\begin{proof}
The presentations \eqref{na:Ksplit},\eqref{na:Knonsplit} have
$K=\langle A,c\rangle$ and $c^A=c^{Q^2}$.
Thus $K'=\langle c^{Q^2-1}\rangle$.
Further commutation with $c$ is trivial, and commutation with
$A$ multiplies the exponent by $Q^2-1$.
Inductively,
\[
 \gamma_j(K)=\langle c^{(Q^2-1)^{j-1}}\rangle\qquad(j\ge2).
\]
Since $|c|=2^h$, $\gamma_{d+1}(K)=1$ exactly when $dv\ge h$.
The least such $d$ is the displayed ceiling.
The assumption $Q^2\ne1\pmod M$ ensures $v<h$, so this is
a nonabelian class of at least two.

For example, take $e=8,f=6$ and $Q=3$ in $\mathcal C_0$.
Here $M=128,N=32$, and $3^{32}=1\pmod{128}$ by
\eqref{na:qval}, while $v_2(3^2-1)=3$.
The class is $\lceil7/3\rceil=3$, proving the last assertion.
\end{proof}
\subsection{Counting the canonical multipliers}
\begin{lemma}[Number of admissible multipliers]\label{na:qcount}
For $M=2^{e-1}\ge8$ and $j\ge1$,
\[
 \#\{Q\in\Z_M:Q\equiv3\pmod4,\ Q^{2^j}=1\}
      =2^{\min(e-3,j)}.
\]
Exactly two such multipliers have $Q^2=1$, namely
$-1$ and $-1+M/2$.
For each remaining multiplier in a canonical type, there are
exactly $N$ preferred marked triples: $N/2$ choices of odd $k$
and two root lifts $\xi$.
\end{lemma}
\begin{proof}
By \eqref{na:qval}, the power condition is equivalent to
$Q=-1\pmod{2^{e-1-j}}$ when that modulus exceeds four.
Together with $Q\equiv-1\pmod4$, it is
\[
 Q=-1\pmod{2^{\max(2,e-1-j)}}.
\]
The number of these residues modulo $2^{e-1}$ is
$2^{e-1-\max(2,e-1-j)}=2^{\min(e-3,j)}$.
If the modulus has already dropped to four, this simply counts
all residues congruent to three modulo four.
The assertion includes $Q=-1$ by direct evaluation.

The roots of $Q^2=1$ modulo a power of two at least eight are
$\pm1,\ \pm1+M/2$.
This follows from $(Q-1)(Q+1)$: the two factors have greatest
common divisor two, so one must be zero modulo $M/2$.
The two roots congruent to three modulo four are precisely
the two listed.
Finally Lemma~\ref{na:root} and Proposition~\ref{na:unique} show that
each allowed multiplier gives exactly the claimed number
of distinct triples.
\end{proof}

\section{The abelian Frattini case: external symmetries and circularity}\label{abelianexternal}
\subsection{Face orders and genera}
The \emph{face order} of a regular dessin is the length of a cycle of
the face permutation on its edge labels. Its geometric face boundary
has twice that many edge sides. Thus a face order of two means a
quadrilateral boundary, not a digon.

\begin{proposition}\label{faces}
For a six-residue datum, put
\begin{equation}\label{zcoords}
 z=(1+p+Pq,\ 1+Qq)\in\Z_M\times\Z_N.
\end{equation}
Then the face order and genus are
\begin{align}
 \ell&=|ab|=2\,\ord(z),\label{faceorder}\\
 \ord(z)&=\operatorname{lcm}\left(
 \frac{M}{\gcd(M,1+p+Pq)},\
 \frac{N}{\gcd(N,1+Qq)}\right),\label{zorder}\\
 g&=1+\frac{4MN-2M-2N-4MN/\ell}{2}.\label{genus}
\end{align}
Here a zero coordinate contributes order one; equivalently
$\gcd(L,0)=L$.
\end{proposition}
\begin{proof}
In the edge-label model of Proposition~\ref{dessindictionary},
the vertex permutations are left multiplication by $a$ and by $b$.
Their right-composed product is left multiplication by $ba$.
The inverse of this product is a possible face-permutation convention;
inversion does not change cycle lengths, and $ba$ is conjugate to $ab$.
Thus the face order is $|ab|$.

Formula \eqref{squareab} at $w=0$ gives
$(ab)^2=x+y+t^\alpha$, whose coordinates are exactly \eqref{zcoords}.
Since $ba=(ab)^a$, its square is $z^\alpha$ and has the same
order. The displayed square formula is for $ab$ itself.
The coset $abK$ has order two and is nontrivial, so $ab$ has even
order. For an element of even order, the order of its square is
half its order. This proves \eqref{faceorder}.
In an additive cyclic group $\Z_L$, the order of a residue $r$ is
$L/\gcd(L,r)$: this is the least positive $k$ such that $L\mid kr$.
The order in a direct product is the least common multiple of the
two coordinate orders, giving \eqref{zorder}.

There are $2M+2N$ vertices and $4MN$ edges.
Regularity makes every face cycle have length $\ell$ on the
$4MN$ edge labels, so there are $4MN/\ell$ faces.
Substituting these values into
$2-2g=V-E+F$ gives \eqref{genus}.
\end{proof}

\begin{corollary}[Faces of family \textup{(I)}]\label{Ifaces}
For the sixteen data of family \textup{(I)},
\[
 z=((\lambda+\rho)M/2,\ (\mu+\sigma)N/2).
\]
Four data have face order two and twelve have face order four.
\end{corollary}
\begin{proof}
Using \eqref{PS},
$P=-\rho(M/2)p$. Since $p$ is odd, this is $\rho M/2$ modulo $M$.
Also $q$ is odd, so $Pq=\rho M/2$.
Thus the first coordinate of \eqref{zcoords} is
$(\lambda+\rho)M/2$.
For the second, expand
\[
 1+(-1+\sigma N/2)(1+\mu N/2)
   =(\sigma-\mu)N/2+\sigma\mu N^2/4.
\]
The last term vanishes modulo $N$, and a half-modulus has order two,
so replacing $\sigma-\mu$ by $\sigma+\mu$ does not change the residue.
The coordinate is therefore $(\sigma+\mu)N/2$.
Both coordinates vanish exactly when $\lambda=\rho$ and $\mu=\sigma$.
There are four such choices. Each of the other twelve choices gives
a nonzero element of order two, proving the stated face orders.
\end{proof}

\subsection{Mirror images and reflexibility}
A dessin is \emph{reflexible} if it has a colour-preserving symmetry
reversing the orientation of the supporting surface, or equivalently
if it is isomorphic to its mirror image. This allows a symmetry
outside the orientation-preserving group $G$ used in the classification.
The mirror triple is $(G,a^{-1},b^{-1})$.

\begin{proposition}[The reflexibility criterion]\label{reflexcriterion}
For any admissible six-residue datum with $M,N\ge2$, the dessin is
reflexible if and only if
\begin{equation}\label{reflexnorm}
 t^{\alpha\beta}=t^{-1}.
\end{equation}
Equivalently its coordinates satisfy
\begin{equation}\label{reflexcoordinates}
 P(Rq-1)=0\pmod M,\qquad S(p+1)=0\pmod N.
\end{equation}
In the square basis belonging to the inverse generators, the mirror
operation on six-residue data is
\begin{equation}\label{mirrordata}
 (P,Q,R,S,p,q)\longmapsto
 \bigl(P,Q,R,S,-R(p+Pq),-Sp-Qq\bigr),
\end{equation}
where the last two entries are reduced modulo $M,N$, respectively.
\end{proposition}
\begin{proof}
If an automorphism $\theta$ sends $a$ to $a^{-1}$ and $b$ to
$b^{-1}$, then it sends $x,y$ to $x^{-1},y^{-1}$.
Its restriction to $K$ is therefore inversion, written additively as
$(u,v)^\delta=(-u,-v)$. This commutes with every automorphism of
the abelian group $K$. Moreover conjugation by $a^{-1}$ on $K$
is $\alpha^{-1}=\alpha$, and similarly for $b^{-1}$.
Consequently the square and action relations in the presentation
are automatically respected by the proposed substitution.

Its effect on the remaining collection relation is decisive.
We have
\[
 b^{-1}a^{-1}ba=[b,a]=t^{\alpha\beta},
 \qquad
 b^{-1}a^{-1}=t^{\alpha\beta}a^{-1}b^{-1}.
\]
Applying the substitution to $ba=t\,ab$ instead gives
$b^{-1}a^{-1}=t^{-1}a^{-1}b^{-1}$.
Thus the relation is preserved exactly when \eqref{reflexnorm}
holds. When it holds, all presentation relations are preserved,
so the substitution defines an endomorphism. Applying it twice
fixes $a,b$; it is therefore an involutory automorphism. Conversely,
an automorphism inverting $a,b$ must satisfy this relation.
Proposition~\ref{dessindictionary} now proves the criterion.

For the coordinates, \eqref{A4} gives $PS=0$ in both moduli,
and \eqref{coaction} yields
\[
 t^{\alpha\beta}=(Rp+RPq,\ Sp+Qq).
\]
The first coordinate of $t+t^{\alpha\beta}$ is
$(1+R)p+RPq=P(Rq-1)$ by \eqref{A7}.
The second is $Sp+(1+Q)q=S(p+1)$ by \eqref{A6}.
This proves \eqref{reflexcoordinates}.
For the mirror, the new squares are $x^{-1},y^{-1}$, and the new
collection coefficient is $t^{\alpha\beta}$, as computed above.
Its coordinates in that inverse square basis are the negatives
of its old coordinates. The action coordinates $P,Q,R,S$ are
unchanged, giving \eqref{mirrordata}.
\end{proof}

\begin{theorem}[All classes in this case are reflexible]
\label{allreflexible}
Every nonmetacyclic complete regular dessin with abelian Frattini subgroup is reflexible.
Consequently the number of reflexible classes of bidegree $(e,f)$
is $\Bab(e,f)$, and there are no chiral classes in the nonmetacyclic
abelian-Frattini case.
\end{theorem}
\begin{proof}
It suffices to verify \eqref{reflexcoordinates} on the exhaustive
lists of Theorem~\ref{largefamilies} and Theorem~\ref{boundary}.
An identity useful for this purpose is
\begin{equation}\label{reflexsimplification}
 P(Rq-1)=-(1+R)p(q-1)\pmod M.
\end{equation}
Indeed $P=-(1+R)p$, and $R^2=1$ implies
$(1+R)R=1+R$ modulo $M$.

In family (I), $P=\rho M/2$ and $S=\sigma N/2$.
Both $Rq-1$ and $p+1$ are even. Multiplying them by the respective
half-moduli gives zero, so both reflexibility equations hold.

In family (II), $q-1$ is divisible by $\max(M,N)/2$ and
$1+R$ is even, proving the first equation by
\eqref{reflexsimplification}. Also $S=\sigma N/2$ if $M\le N$,
whereas $S=0$ if $M=2N$. As $p+1$ is even, the second equation
holds in both cases. Family (III) consists of reciprocals of
family (II). A generator-inverting automorphism for $(G,a,b)$
also inverts the generators of $(G,b,a)$, so reflexibility is
preserved by reciprocity; this proves the assertion for (III).

For the boundary lists, if $M=N=2$ then $P=S=0$.
If $N=2,M\ge4$, then $S=0$ and $P$ is either $0$ or $M/2$;
again $Rq-1$ is even. If $M=N=4$, then $P,S$ are even and
each multiplier in \eqref{reflexcoordinates} is even, giving
zero modulo four. For the first sixteen data with $N=4,M\ge8$,
we have $P\in\{0,M/2\}$ and $p+1$ divisible by $M/2$, hence by
four. These imply the two congruences. In the eight extra data
at $M=8,N=4$, $q=1$ and $S=0$, so both congruences follow
immediately from \eqref{reflexsimplification}.
Their reciprocal lists are reflexible as well. Every possible
nonmetacyclic datum has now been covered.
\end{proof}

The generator-inverting automorphism is unique, because $a,b$
generate $G$. It has order two, since $|a|,|b|\ge4$ in this
case. Thus the full colour-preserving symmetry group, allowing
orientation reversal, is a split extension $G\rtimes C_{\mathrm{ab},2}$ of
order $2|G|=8MN$. To see the splitting geometrically, choose the
orientation-reversing symmetry that fixes the distinguished edge.
It reverses both vertex rotations; its square fixes the edge and
preserves the orientation, and is therefore the identity.
\subsection{A coordinate criterion for circularity}
Circularity is stricter than a bound on the face order. For example,
a face can have eight sides and revisit a vertex. The following
criterion detects the repeated vertices exactly.

\begin{proposition}[The circularity criterion]\label{circularcriterion}
Write $z=(z_1,z_2)$ as in \eqref{zcoords}, and put
\[
 d_1=\frac{M}{\gcd(M,z_1)},\qquad
 d_2=\frac{N}{\gcd(N,z_2)}.
\]
For any admissible six-residue datum, the dessin is circular if and
only if $d_1=d_2$. Equivalently,
\begin{equation}\label{circularcoordinates}
 \frac{M}{\gcd(M,1+p+Pq)}
 =\frac{N}{\gcd(N,1+Qq)}.
\end{equation}
If these equal orders are $d$, every face boundary is a simple
cycle of length $4d$.
\end{proposition}
\begin{proof}
By Proposition~\ref{dessindictionary}, circularity is equivalent
to trivial intersections of $\langle ab\rangle$ with both vertex
rotation subgroups.

The image of $ab$ in $G/K$ is the sum of the two nonzero basis
vectors. Its even powers lie in $K$ and its odd powers do not.
If a power of $ab$ lies in $\langle a\rangle$ or $\langle b\rangle$,
comparison in $G/K$ forces that power to be even and forces the
power of $a$ or $b$ to be even as well. Since $(ab)^2=z$, it follows
that
\[
 \langle ab\rangle\cap\langle a\rangle
      =\langle z\rangle\cap\langle x\rangle,
 \qquad
 \langle ab\rangle\cap\langle b\rangle
      =\langle z\rangle\cap\langle y\rangle.
\]
The cyclic subgroup $\langle z\rangle$ has order
$d=\operatorname{lcm}(d_1,d_2)$.
Its projection to $\langle y\rangle$ has image of order $d_2$
and kernel $\langle z\rangle\cap\langle x\rangle$, so that
intersection has order $d/d_2$. The other intersection similarly
has order $d/d_1$. Both are trivial exactly when $d_1=d_2=d$.
Finally the face order is $2d$ by Proposition~\ref{faces}, and
the boundary has twice that many sides, giving $4d$.
\end{proof}
\subsection{The circular subfamilies}
The coordinate-order test now isolates the circular dessins without
any further group isomorphism tests.

\begin{theorem}[Circular data in the three large families]
\label{circularfamilies}
Assume $M,N\ge8$. The circular data in
Theorem~\ref{largefamilies} are exactly the following.
\begin{enumerate}[label=\textup{(\Roman*)},leftmargin=2.5em]
\item In family (I), require
\begin{equation}\label{circularI}
 \lambda+\rho=\mu+\sigma\pmod2.
\end{equation}
There are eight circular data, four with face order two and
four with face order four.
\item In family (II), if $M\le N$, write $q=1+\mu N/2$.
Require
\begin{equation}\label{circularII}
 p=\begin{cases}
 R\pmod M,&\mu=\sigma,\\
 R(1-M/2)\pmod M,&\mu\ne\sigma.
 \end{cases}
\end{equation}
There are eight circular data, four of each face order two and
four. If $M=2N$, require $p=R$; there are two circular data,
both with face order two. For $M>2N$ this family is absent.
\item In family (III), if $N\le M$, write $p=-1+\lambda M/2$.
Require
\begin{equation}\label{circularIII}
 q=\begin{cases}
 -Q\pmod N,&\lambda=\rho,\\
 Q(N/2-1)\pmod N,&\lambda\ne\rho.
 \end{cases}
\end{equation}
There are eight circular data, four of each face order two and
four. If $N=2M$, require $q=-Q$; there are two circular data,
both with face order two. For $N>2M$ this family is absent.
\end{enumerate}
In particular every circular dessin in these families has
quadrilateral or octagonal face boundaries.
\end{theorem}
\begin{proof}
In family (I), Corollary~\ref{Ifaces} gives
$z=((\lambda+\rho)M/2,(\mu+\sigma)N/2)$.
Each coordinate has order one when its binary sum is even and
order two when its binary sum is odd. Thus their orders agree
exactly in \eqref{circularI}. Both vanish in four choices; both
are nonzero in four choices. Proposition~\ref{circularcriterion}
and the face-order formula give the assertion.

In family (II) with $M\le N$, direct substitution gives
\[
 z_1=1-Rp\pmod M,
 \qquad z_2=(\sigma+\mu)N/2\pmod N.
\]
For the first equality, the extra term
$(1+R)p\mu N/2$ is divisible by $M$, since $1+R$ is even and
$M\mid N$. For the second, expand $1+Qq$ and use $N\mid N^2/4$.
If $\sigma=\mu$, the second coordinate is zero, so circularity
requires $z_1=0$. As $R^{-1}=R$, this means $p=R$.
If $\sigma\ne\mu$, the second coordinate has order two, so
$z_1$ must be the unique residue of order two modulo $M$, namely
$M/2$. This gives $p=R(1-M/2)$. All these residues are odd.
The two choices of $\rho$ and the four choices of $(\sigma,\mu)$
give eight data, half of each face order.
If $M=2N$, then $\sigma=0,q=1$ and $z=(1-Rp,0)$.
Only $p=R$ is circular, one choice for each of the two values of
$\rho$, and both have face order two.

In family (III) with $N\le M$, substitution gives
\[
 z_1=(\lambda+\rho)M/2\pmod M,
 \qquad z_2=1+Qq\pmod N.
\]
When $\lambda=\rho$, the first coordinate is zero, so the second
must vanish: $q=-Q$, using $Q^{-1}=Q$.
When $\lambda\ne\rho$, the first coordinate has order two, so
$z_2=N/2$, giving $q=Q(N/2-1)$. Again all chosen residues are
odd. The binary choices give four data with zero $z$ and four
with both coordinates nonzero of order two.
If $N=2M$, then $\rho=0,p=-1$ and $z_1=0$; only $q=-Q$
is circular, for the two values of $\sigma$. These have face
order two. The nonexistence statements are those of
Theorem~\ref{largefamilies}. This covers every datum in the
three families.
\end{proof}

For completeness, the boundary lists admit the following equally
explicit circular restrictions. Reciprocal lists are obtained by
exchanging $a,b$ and the moduli as in \eqref{swap}.
\begin{enumerate}[label=\textup{(\alph*)},leftmargin=2.5em]
\item For $M=N=2$, the unique datum is circular, with face order two.
\item For $N=2,M\ge4$, require $p=R$. Exactly two data are
circular, both with face order two.
\item For $M=N=4$, require
$p=R$ if and only if $q=-Q$. Exactly eight data are circular,
four with face order two and four with face order four.
\item For the first sixteen data with $N=4,M\ge8$, require
$\lambda=\rho$ if and only if $q=-Q\pmod4$.
There are eight circular data, four of each face order.
For the additional eight data at $M=8,N=4$, require $p=R$.
There are two further circular data, both with face order two.
\end{enumerate}
Here is a verification of all four restrictions. In (a), $z=0$.
In (b), $z=(1-Rp,0)$, so it must vanish, giving $p=R$.
In (c), each coordinate is zero or two modulo four.
The first is zero exactly when $p=R$, and the second exactly
when $q=-Q$. Equal coordinate orders therefore give the stated
biconditional. For each of the four pairs $(R,Q)$, one choice
has both coordinates zero and one has both nonzero.
In the first part of (d), $z_1=(\lambda+\rho)M/2$ and
$z_2=1+Qq\pmod4$, so the same comparison proves the rule.
In the extra data $z=(1-Rp,0)$; each of the two roots $R$
permits exactly the one circular value $p=R$.

\subsection{Enumeration of circular dessins}
Let $C_{\mathrm{ab},2}(e,f)$ and $C_{\mathrm{ab},4}(e,f)$ denote the numbers of circular
classes in this case with face orders two and four, respectively.
Thus their geometric face boundaries have four and eight sides.
Put $C_{\rm ab}(e,f)=C_{\mathrm{ab},2}(e,f)+C_{\mathrm{ab},4}(e,f)$.

\begin{theorem}[Circular counts and genera]\label{circularcounts}
There are no circular classes in this case with any other face
order. For $e\ge f\ge1$, the complete counts are
\begin{center}
\begin{tabular}{lrrr}
\toprule
range&$C_{\mathrm{ab},2}(e,f)$&$C_{\mathrm{ab},4}(e,f)$&$C_{\rm ab}(e,f)$\\
\midrule
$f\le1$&0&0&0\\
$e=f=2$&1&0&1\\
$f=2<e$&2&0&2\\
$e=f=3$&4&4&8\\
$f=3,\ e=4$&6&4&10\\
$f=3,\ e\ge5$&4&4&8\\
$e=f\ge4$&12&12&24\\
$f\ge4,\ e=f+1$&10&8&18\\
$f\ge4,\ e\ge f+2$&8&8&16\\
\bottomrule
\end{tabular}
\end{center}
For $e<f$ exchange the two exponents. Every class counted here
is reflexible. Its genus is
\begin{equation}\label{circulargenera}
 g=\begin{cases}
 1+MN-M-N,&\ell=2,\\
 1+3MN/2-M-N,&\ell=4.
 \end{cases}
\end{equation}
\end{theorem}
\begin{proof}
The boundary restrictions just proved give the first six rows.
In the large ranges family (I) always contributes four classes
of each face order. When $M=N$, families (II) and (III) each
contribute four of each, giving $(12,12)$.
When $M=2N$, family (II) contributes $(2,0)$ and family (III)
contributes $(4,4)$; adding family (I) gives $(10,8)$.
When $M\ge4N$, family (II) is absent and family (III) contributes
$(4,4)$, giving $(8,8)$. These ranges translate exactly into the
last three rows. The classification lists are disjoint and
nonredundant, so no further identifications are made.
All circular data in the proof have $\ord(z)=1$ or $2$, proving
the exclusion of other face orders. Reflexibility is
Theorem~\ref{allreflexible}. Finally substituting $\ell=2,4$
into \eqref{genus} gives \eqref{circulargenera}.
\end{proof}

\subsection{The symmetry criterion}
A dessin is \emph{symmetric} here if it is isomorphic to its reciprocal
by an orientation-preserving map. This requires $e=f$.

\begin{proposition}\label{symmetry}
Suppose $M=N$. A datum admits a colour-exchanging symmetry if and only if
\begin{equation}\label{symmetrycondition}
 P=S,\qquad Q=R,\qquad p+q=0\pmod M.
\end{equation}
For nonmetacyclic data, the number $S(e)$ of symmetric classes is
\[
 S(2)=1,\qquad S(e)=4\quad(e\ge3).
\]
Consequently the number of classes up to reciprocality on the diagonal
is $(\Bab(e,e)+S(e))/2$ for $e\ge2$.
\end{proposition}
\begin{proof}
Choose a colour-exchanging isomorphism. By regularity, we can compose
it with an automorphism preserving both colors so that the chosen
edge is fixed. The resulting map must interchange the two
distinguished rotations.
Thus symmetry is equivalent to an automorphism $\tau$ of $G$ with
$a^\tau=b$ and $b^\tau=a$.

Its restriction $\gamma$ to $K$ must satisfy
$x^\gamma=y$, $y^\gamma=x$, or
$(u,v)^\gamma=(v,u)$. Compatibility with right conjugation requires
\begin{equation}\label{rightsymmetry}
 \alpha\gamma=\gamma\beta,\qquad
 \beta\gamma=\gamma\alpha.
\end{equation}
For example,
$(u^\alpha)^\gamma=(u^a)^\tau=(u^\tau)^b=(u^\gamma)^\beta$,
which explains the order in the first identity.
Coordinate calculation gives
\[
 (u,v)^{\alpha\gamma}=(Qv,u+Pv),\qquad
 (u,v)^{\gamma\beta}=(Rv,Sv+u).
\]
Hence \eqref{rightsymmetry} is equivalent to $Q=R,P=S$.
Applying the exchange to $ba=t\,ab$ also requires
$ab=t^\gamma ba$, or $t^\gamma=t^{-1}$.
In coordinates this says $(q,p)=(-p,-q)$, giving $p+q=0$.
These prove necessity.

Conversely, these conditions show directly that interchanging $a,b$
and $x,y$ preserves every relation in \eqref{overviewpresentation}.
It therefore defines a homomorphism $\tau:G\to G$.
Its square fixes every generator, so $\tau^2=\id$.
Thus it is an automorphism and gives the required symmetry.

For the count, the $M=2$ list has just one datum and it is fixed
by \eqref{swap}. For $M=4$, choose $Q=R$ in two ways and choose
odd $p$ in two ways, putting $q=-p$. Then
$S=(1+R)(-p)=P$, so there are four fixed data.
For $M\ge8$, a symmetric datum has $Q=R$.
Families (II) and (III) have different signs for these two roots,
so cannot be symmetric. In family (I), equality of the roots gives
$\rho=\sigma$, and $q=-p$ gives $\mu=\lambda$.
There are two choices for each of $\rho,\lambda$, hence four.
Finally an involution on a finite set has one-element orbits at its
fixed points and two-element orbits elsewhere. Its number of orbits
is $S(e)+(\Bab(e,e)-S(e))/2$, proving the last formula.
\end{proof}

\begin{corollary}[Symmetric circular classes and reciprocality]
\label{symmetriccircular}
Every symmetric dessin in this case is circular and reflexible.
For $e=f$, the number of circular classes up to reciprocality is
\[
 \frac{C_{\rm ab}(e,e)+S(e)}2=
 \begin{cases}
 1,&e=2,\\
 6,&e=3,\\
 14,&e\ge4.
 \end{cases}
\]
For $e>f$, fixing the larger exponent as the first coordinate gives
$C_{\rm ab}(e,f)$ circular reciprocal pairs. Identifying orientation reversal
does not reduce any of these counts, since all classes are reflexible.
\end{corollary}
\begin{proof}
For $M=N=2$ there is only the circular datum. For $M=N=4$,
symmetry gives $R=Q$ and $q=-p$. The circular boundary condition
is $p=R$ if and only if $q=-Q$, which is then automatic.
For $M=N\ge8$, the symmetric data lie in family (I) with
$\rho=\sigma$ and $\lambda=\mu$. These equalities imply
\eqref{circularI}, so all four symmetric classes are circular.
Reflexibility is Theorem~\ref{allreflexible}.

Reciprocity preserves circularity, and its fixed points on the
diagonal are exactly the $S(e)$ symmetric classes. Its remaining
orbits have size two, so the circular orbit count is
$S(e)+(C_{\rm ab}(e,e)-S(e))/2$. Substituting the counts already proved
gives the three displayed values. When $e>f$, reciprocity exchanges
the two ordered bidegrees, so choosing the larger exponent first
selects one element of each orbit. Finally a reflexible class is
already isomorphic to its mirror as an oriented dessin; allowing
orientation reversal introduces no new identification.
\end{proof}

\subsection{Three further face distributions}
The following table records the number of data of each face order.
It counts only the nonmetacyclic abelian Frattini case.
\[
\begin{array}{c|r|rrr}
\text{underlying graph}&B&\ell=2&\ell=4&\ell=8\\ \hline
K_{8,8}&16&4&12&0\\
K_{16,8}&24&6&14&4\\
K_{16,16}&80&12&36&32
\end{array}
\]
All the displayed classes are reflexible. The numbers of circular
classes in these three rows are respectively $8,10,24$; their
quadrilateral and octagonal subcounts are given by
Theorem~\ref{circularcounts}. A face order of four does not by
itself imply circularity: family (I), for example, has twelve
such classes but only four of them are circular.
Here are explicit checks of the face-order entries.

For $K_{8,8}$, $M=N=4$ and the four odd choices are independent.
Using \eqref{PS}, the first coordinate of $z$ becomes
$1+p-(1+R)pq$. If $R=3$, it is zero exactly when $p=3$.
If $R=1$, it is zero exactly when $p=1$, independently of $q$.
The second coordinate $1+Qq$ is zero exactly when $q=-Q$.
There are therefore two choices of $(R,p)$ and two of $(Q,q)$
making $z=0$, giving four face orders two.
All other coordinates are zero or two, so the remaining twelve
have face order four.

For $K_{16,8}$, $M=8,N=4$. In the first sixteen boundary data,
$R=-1+4\rho$ and $p=-1+4\lambda$.
The first coordinate of $z$ is $4(\lambda+\rho)$,
and the second vanishes exactly when $q=-Q$.
Thus four of these data have $z=0$, and twelve have order-two $z$.
In the eight additional data, $R=1+4\rho,Q=3,q=1$.
Now $P=-(2+4\rho)p$, so
\[
 z=(1-(1+4\rho)p,\ 0).
\]
As $p$ ranges over the four odd residues modulo eight, multiplication
by $1+4\rho$ permutes those residues. For each $\rho$, this coordinate
takes the values $0,2,4,6$ exactly once.
Hence these eight data contribute two face orders two, two face
orders four and four face orders eight, giving the stated totals.

For $K_{16,16}$, $M=N=8$. Family (I) gives four face orders two
and twelve face orders four by Corollary~\ref{Ifaces}.
For family (II), $R=1+4\rho,Q=-1+4\sigma$, $q=1+4\mu$,
and $p$ is any odd residue. Direct substitution gives
\[
 z=(1-(1+4\rho)p,\ 4(\sigma+\mu)).
\]
For each $\rho$, the first coordinate runs through $0,2,4,6$.
Among the four pairs $(\sigma,\mu)$, two make the second coordinate
zero and two make it four.
Consequently family (II) contributes four face orders two,
twelve face orders four, and sixteen face orders eight.
Family (III) has the same distribution: generator exchange
\eqref{swap} interchanges families (II) and (III), and it preserves
$|ab|$ because $|ba|=|ab|$. Adding the three distributions proves
the last row.

\subsection{Examples}
\begin{example}[The unique datum on $K_{4,4}$]
Here $M=N=2$, and the single nonmetacyclic datum has
\[
 (P,Q,R,S,p,q)=(0,1,1,0,1,1).
\]
The square generators $x=a^2,y=b^2$ are central involutions, and
$ba=xyab$. Since $t=xy$ is central,
$[b,a]=xy$ as well in this particular example.
The sixteen normal forms of Theorem~\ref{construction} make the
order and exactness explicit.

The nonmetacyclicity obstruction is especially transparent.
The squares of primitive elements are $x,y,1$, according as their
Frattini cosets are $Ka,Kb,Kab$. The derived subgroup is
$\langle xy\rangle$, contained in none of the corresponding cyclic
square subgroups. Thus Lemma~\ref{primitive} applies directly.
The coordinate $z$ is zero, so $\ell=2$.
There are eight vertices, sixteen edges and eight faces,
giving genus one. It is both reflexible and circular.
\end{example}
\begin{example}[All four data on $K_{8,4}$]
Let $M=4,N=2$. The four tuples and their topological invariants are
shown below.
\[
\begin{array}{c|rrrrrr|rr}
 &P&Q&R&S&p&q&\ell&g\\ \hline
1&0&1&3&0&1&1&4&7\\
2&0&1&3&0&3&1&2&3\\
3&2&1&1&0&1&1&2&3\\
4&2&1&1&0&3&1&4&7
\end{array}
\]
To see the calculations, the boundary list permits $R=3,1$ and
$p=3,1$, with $Q=q=1$. If $R=3$, then
$P=-(1+R)p=0\pmod4$; if $R=1$, then $P=-2p=2\pmod4$.
Also $S=(1+Q)q=0\pmod2$.
Formula \eqref{zcoords} gives $z=(1+p+P,0)$, which is either zero
or $(2,0)$. The corresponding face orders are two and four.
Finally \eqref{genus}, with $4MN=32$ and $2M+2N=12$, gives
genus three or seven respectively.

All four dessins are reflexible. The second and third are
circular, and the first and fourth are not, since the second
coordinate of $z$ is zero while its first is nonzero in the
first and fourth rows.

The third datum also illustrates why the collection coefficient is
not generally the commutator. Here $t=(1,1)$ and
$t^{\alpha\beta}=(3,1)$, so
\[
 ba=xy\,ab,\qquad [b,a]=x^3y .
\]
The parity criterion agrees for both coordinate pairs, but replacing
one by the other in the collection relation would be incorrect.
\end{example}
\section{The nonabelian Frattini case: external symmetries and circularity}\label{nonabelianexternal}
The notation in this section is that of Theorem~\ref{na:main}.
In particular the root parameter is $\xi$, and the preferred pair
$(r,s)$ is specified in \eqref{na:preferred}. Put
\[
 h=e-1,\quad n=f-1,\quad M=2^h,\quad N=2^n,
 \quad \theta=\xi-R=1+Q\xi\pmod M.
\]
All statements are unchanged by the additional reciprocal marking
on the diagonal.
\subsection{Face orders and genera}
\begin{proposition}[Computing the face order]\label{na:faces}
In preferred ordering put
\[
 j=k+1+\varepsilon N/2,\qquad t=\xi-R,\qquad A=r^2 .
\]
Then $(rs)^2=A^j c^t$.
Let $d$ be the least positive power of two for which
\[
 N\mid jd
\]
and
\[
 \begin{cases}
 t\,\Nrm_d(Q^{2j})=0\pmod M,
       &\mathcal C_0,\mathcal L_0,\\
 (jd/N)(M/2)+t\,\Nrm_d(Q^{2j})=0\pmod M,
       &\mathcal C_1.
 \end{cases}
\]
The face order is $\ell=2d$, and the genus is
\[
 g=1+\frac{4MN-2M-2N-4MN/\ell}{2}.
\]
A face order $\ell$ corresponds to a boundary of $2\ell$
edge sides.
\end{proposition}
\begin{proof}
Formula \eqref{na:facesquare} gives the square, and $w=A^{N/2}$
absorbs its extra factor in $\mathcal L_0$.
The power formula in $K$ gives
\[
 (A^j c^t)^d=A^{jd}c^{t\,\Nrm_d(Q^{2j})}.
\]
In the split types the normal form is one exactly when the
$A$ exponent vanishes modulo $N$ and the $c$ exponent vanishes
modulo $M$. In the nonsplit type, an $A$ exponent of $jd$
with $N\mid jd$ contributes $(jd/N)(M/2)$ to the $c$ exponent.
This proves the tests. The order of an element of a finite
$2$-group is a power of two, so restricting $d$ to such powers
does not omit the true order.

The coset of $rs$ in $G/K$ is nontrivial of order two.
Hence $|rs|$ is twice the order of its square.
The face permutation in the edge-label construction has cycle
length $|ab|$; this equals $|rs|$ because $ab$ is $rs$ or $sr$
and the latter two are conjugate.
There are $2M+2N$ vertices, $4MN$ edges and $4MN/\ell$ faces.
Euler's formula gives the genus.
\end{proof}

\begin{example}
Take $(e,f)=(5,3)$, so $M=16,N=4$. Only type $\mathcal C_0$ occurs.
The noninvolutory multipliers congruent to three modulo four
are $Q=3,11$, and the odd $k$ are one and three.
The four choices of $(Q,k)$ give these root lifts:
\[
\begin{array}{c|c|c|c|c}
 Q&k&\xi&R&\text{face orders for the two lifts}\\ \hline
 3&1&1,\ 9&13&8,\ 8\\
 3&3&5,\ 13&5&2,\ 4\\
 11&1&1,\ 9&5&8,\ 8\\
 11&3&5,\ 13&13&4,\ 2
\end{array}
\]
For instance, $Q=3,k=1$ gives
$F(Z)=Z^2+Z-2$ modulo eight, whose unique odd root is one.
Its two lifts are $1,9$ and $R=-2\xi-1=13$ modulo sixteen.
The corresponding marked presentations have $b^8=c^{16}=1$,
$[b,a]=c$, $c^b=c^3$, $c^a=c^{13}$ and
$a^2=b^2c^\xi$, with $\xi=1$ or nine.

For this row $A=b^2$ has order four, and
$(ba)^2=A^2c^{\xi-13}$. Its square is $c^8$, and its fourth
power is one, so the face order is eight.
For the rows with $k=3$, $A^{k+1}=A^4=1$; the two square
coordinates are one and $c^8$, giving face orders two and four.
This checks the entire distribution:
two classes have face order two, two have face order four
and four have face order eight.
Their genera, from Proposition~\ref{na:faces}, are respectively
$45,77,93$. By Theorem~\ref{nacircularfamilies}, exactly the two
face-order-two classes are circular: $(Q,k,\xi)=(3,3,5)$ and
$(11,3,13)$. Formula~\eqref{namirror} interchanges them.
The other six face boundaries repeat vertices, including both
classes of face order four.

In all eight cases $Q^2=9\pmod{16}$, so
\[
 K=\langle A,c\mid A^4=c^{16}=1,\ c^A=c^9\rangle.
\]
Its derived subgroup is $\langle c^8\rangle$ of order two,
central in $K$; thus it has nilpotency class two.
The class formula gives the same answer, $\lceil4/3\rceil=2$.

The separate and combined counts in a few small cases are
\[
\begin{array}{c|r|r|r}
 (e,f)&\Bab&\Bna&T\\ \hline
 (4,3)&24&0&24\\
 (5,3)&16&8&24\\
 (5,4)&64&48&112\\
 (6,4)&48&64&112\\
 (5,5)&144&128&272
\end{array}
\]
The two contributions are kept separate because their reflection
behavior differs, as proved below.
\end{example}

\subsection{Mirror images and reflexibility}
\begin{theorem}[Mirror parameters]\label{nonabelmirror}
The mirror of the canonical datum $(Q,k,\xi,R)$ is the datum
\begin{equation}\label{namirror}
 Q^*=Q^{-1},\qquad k^*=k,\qquad
 \xi^*=-\xi Q R^{-1}\pmod M,\qquad R^*=R^{-1}\pmod M,
\end{equation}
in the same family and with the same choice of preferred ordering.
None of the nonabelian Frattini dessins is reflexible. They occur
in $\Bna(e,f)/2$ chiral pairs.
\end{theorem}
\begin{proof}
Use the reflected pair $r^*=r^{-1},s^*=s^{-1}$ and put
$c^*=[r^{-1},s^{-1}]$. Directly from the definition of commutators,
\[
 [r^{-1},s^{-1}]=[r,s]^{r^{-1}s^{-1}}.
\]
The actions of $r,s$ commute on the abelian derived subgroup.
In the cyclic-derived cases this gives
$c^*=c^{(QR)^{-1}}$. In $\mathcal L_0$, each inverse action
contributes a factor $w$, and these cancel since their total
exponent is even. The same formula therefore holds there too;
also $(r^*)^N=w$ because $w$ has order two.

Consequently $c^*$ has the same order as $c$, and its actions by
$r^*,s^*$ have multipliers $Q^{-1},R^{-1}$, with the same $w$
factors when required. Moreover
\[
 (s^*)^2=(r^{2k}c^\xi)^{-1}
       =r^{-2k}c^{-\xi Q^{-2k}}.
\]
Expressing the last factor in the basis $c^*=c^{(QR)^{-1}}$
gives $\xi^*=-\xi Q^{1-2k}R=-\xi Q R^{-1}$, using
$R^2=Q^{2k}$. This proves \eqref{namirror}. It also verifies
$R^*=(1-Q^*)\xi^*-1$ directly. The reflected triple is exact
and has the same Frattini subgroup, so it is an admissible datum;
alternatively the relations just computed verify its defining
presentation. The range condition on $Q$ is invariant under
inversion. No change of the residue $k\in\Z_N$ is needed, even
in the nonsplit family, since $r^{-2k}=(r^*)^{2k}$ as elements.

If a mirror were isomorphic to its original marked triple,
marked uniqueness would give $Q=Q^{-1}\pmod M$. This contradicts
$Q^2\ne1\pmod M$. For a diagonal reciprocal entry, reflection
preserves which generator has action $3\pmod4$ and which has
 action $1\pmod4$, so it stays in that entry's preferred ordering.
Thus the uniqueness argument applies to every marking.
An involution with no fixed points partitions the classes into
pairs, proving the count.
\end{proof}

\subsection{Terminal involutions and the valuation of the face square}
In preferred ordering the face square is
\begin{equation}\label{nafacesqnew}
 F=(rs)^2=A^j c^\theta,\qquad A=r^2,\qquad
 j=k+1+\varepsilon N/2,
\end{equation}
where $\varepsilon=1$ only in $\mathcal L_0$.
The letter $F$ in this subsection denotes this group element,
not the root polynomial of \eqref{na:poly}.
Circularity holds exactly when $\langle F\rangle$ meets neither
$\langle A\rangle$ nor $\langle s^2\rangle$ nontrivially.
Indeed an intersection of the face subgroup with a vertex subgroup
must have even exponents in $G/\Phi(G)$, as in the proof of
Proposition~\ref{circularcriterion}.

\begin{lemma}\label{tvaluation}
Let $\lambda=v_2(Q+1)$, let $d=\log_2\ord_M(Q)=h-\lambda$,
and put $i=v_2(k+1)$, where $1\le i\le n$.
Here $2\le d\le h-2$. The two lifts of the odd root give
$\theta$ values differing by $M/2$. If $i<d$, both have valuation
$h-d+i-1$. If $i\ge d$, they are exactly $0,M/2$ modulo $M$.
\end{lemma}
\begin{proof}
The valuation formula \eqref{na:qval} shows that the order of
$Q\equiv3\pmod4$ with $Q^2\ne1$ is $2^{h-\lambda}\ge4$.
The two root lifts differ by $M/2$, and multiplying by odd $Q$
shows that the two $\theta$ values differ by $M/2$ as well.
The canonical root equation gives
\[
 (Q-1)\xi^2+2\xi=\Nrm_{2k}(Q)\pmod M.
\]
Substitution of $\theta=Q\xi+1$, followed by multiplication by
$Q^2$, gives the useful shifted equation
\[
 (Q-1)\theta^2+2\theta=\Nrm_{2(k+1)}(Q)\pmod M.
\]
Thus
\[
 \theta\left(1+\frac{Q-1}{2}\theta\right)
       =\frac{\Nrm_{2(k+1)}(Q)}2\pmod{M/2}.
\]
Since $\theta$ is even, the parenthesized factor is odd.
The norm valuation on the right is $\lambda+i-1$:
write $k+1=2^i u$ with $u$ odd, and the additional norm of odd
length has odd value. Hence the valuation of $\theta$, truncated
at $h-1$, is $\min(h-1,\lambda+i-1)$.
Using $d=h-\lambda$ proves the two alternatives. In the second
alternative the two residues divisible by $M/2$ are $0,M/2$,
so their difference identifies them exactly.
\end{proof}

\begin{lemma}[The three involutions]\label{threeinvolutions}
In the split Frattini presentations \eqref{na:Ksplit}, put
$U=A^{N/2}$ and $V=c^{M/2}$. The nonidentity involutions of $K$
are $U,V,UV$. The terminal involutions of $r,s$ are
\[
 (U,V)\quad(M>N),\qquad (U,UV)\quad(M=N).
\]
If the coordinate orders of $A^j$ and $c^\theta$ are $2^a,2^b$,
then $F$ has order $2^{\max(a,b)}$. Its terminal involution is
$U$ when $a>b$, $V$ when $b>a$, and $UV$ when $a=b>0$.
For $a=b=0$ it is the identity.
\end{lemma}
\begin{proof}
An involution in the normal forms $A^u c^v$ must have
$u=0$ or $N/2$ by its image modulo $\langle c\rangle$.
Since $Q^N=1$ in both split families, $A^{N/2}$ centralizes $c$.
The square condition then forces $v=0$ or $M/2$, giving the list.
The terminal involutions of $r,s$ were computed in
Proposition~\ref{na:exactness}.

For the order calculation,
$(A^j c^\theta)^{2^t}=A^{j2^t}
 c^{\theta\Nrm_{2^t}(Q^{2j})}$.
Since $Q^{2j}=1\pmod8$, its norm has valuation exactly $t$.
Thus its order is the maximum of the two coordinate orders.
At half that order only the coordinate with the larger order
survives, or both do if the orders agree. A surviving coordinate
is its cyclic factor's unique involution. This proves the stated
terminal-involution rule.
\end{proof}

In particular the split off-diagonal circularity test is $a=b$,
including the identity square. On the diagonal it is $b>a$ or
$a=b=0$. This change is essential: on the diagonal the terminal
involution of $s$ is $UV$, not $V$.
\subsection{The circular subfamilies}
\begin{theorem}[All nonabelian Frattini circular data]\label{nacircularfamilies}
The following lists are exhaustive. A condition on $\theta$ selects
one root lift when its two values are $0,M/2$.
\begin{enumerate}[label=\textup{(\alph*)},leftmargin=2.5em]
\item On the diagonal $e=f$, the circular data are
\[
 \mathcal C_0:\ k=N-1,\qquad
 \mathcal L_0:\ k=N/2-1.
\]
For each admissible $Q$ both root lifts are allowed. The lift with
$\theta=0$ has face order $2$, and the lift with $\theta=M/2$
has face order $4$. Include both reciprocal markings.
\item If $e>f$, the circular data of $\mathcal C_0$ are
\begin{align*}
 &k=N-1,\quad\theta=0 &&(\text{face order }2);\\
 &k=N/2-1,\quad Q^{N/2}=1,\quad\theta=M/2
       &&(\text{face order }4);\\
 &\ord_M(Q)=N/2,\quad 1\le i=v_2(k+1)\le n-2
       &&(\text{both lifts, face order }2^{n-i+1}).
\end{align*}
\item If $e>f$, the circular data of $\mathcal L_0$ are
\begin{align*}
 &k=N/2-1,\quad\theta=0 &&(\text{face order }2);\\
 &k=N-1,\quad\theta=M/2 &&(\text{face order }4);\\
 &\ord_M(Q)=N/2,\quad 1\le i=v_2(k+1)\le n-2
       &&(\text{both lifts, face order }2^{n-i+1}).
\end{align*}
The last line in each of \textup{(b),(c)} is empty for $n=2$.
\item In $\mathcal C_1$, the circular data are exactly
$k=N-1$, $\theta=M/2$, with face order $2$.
\end{enumerate}
Here face order $\ell$ means a geometric boundary of $2\ell$ sides.
\end{theorem}
\begin{proof}
For the split families, apply Lemma~\ref{threeinvolutions}.
First suppose $i=v_2(k+1)\le n-2$.
Adding $N/2$ does not change this valuation, so in both split
families $a=n-i\ge2$. If $i\ge d$, Lemma~\ref{tvaluation}
gives $b\le1$, which is never circular. If $i<d$, it gives
$b=d+1-i$. Off the diagonal, $a=b$ is therefore exactly $d=n-1$,
and both lifts are circular. Their face order is $2^{a+1}$,
as in the last lines of (b),(c). On the diagonal, every admissible
multiplier has $d\le h-2=n-2$, so $b\le a-1$ and none of
these data is circular.

There are two remaining odd values of $k$. If $k=N/2-1$, then
$i=n-1$. In $\mathcal C_0$ the $A$ coordinate has order two.
If $d=n$, the other coordinate has order four, which is not
circular. Otherwise $d\le n-1$, or $Q^{N/2}=1$, and its two
coordinate values are $0,M/2$. Off the diagonal choose $M/2$
to make the orders equal; on the diagonal neither lift is
circular. In $\mathcal L_0$, this $k$ gives $j=N$, so the
$A$ coordinate is the identity. All its admissible multipliers
have $d\le n-1$, so the two values are again $0,M/2$.
Off the diagonal only zero is allowed; on the diagonal both are
allowed, by the terminal-involution rule.

If $k=N-1$, then $i=n$ and $\theta=0,M/2$ in both split types.
In $\mathcal C_0$ the $A$ coordinate is the identity; this gives
the first line of (b), and both circular lifts on the diagonal.
In $\mathcal L_0$, $j=3N/2$, so the $A$ coordinate has order two.
Off the diagonal choose $\theta=M/2$; on the diagonal neither
lift is allowed. This proves (a)--(c), including the face orders.

For $\mathcal C_1$, $h=n+1$ and $A^N=V=c^{M/2}$.
If $k\ne N-1$, put $i=v_2(k+1)<n$ and $a=n-i\ge1$.
The valuation lemma gives $v_2(\theta)\ge i+1$ (with the usual
infinite valuation for zero), since $\lambda\ge2$ and $i+1\le h-1$.
In the power $F^{2^a}$ the $A$ part is $A^{N u}=V$ with $u$ odd,
while its $c$ part is the identity, since its valuation is at least
$i+1+a=h$. Hence $F^{2^a}=V\ne1$ and $F^{2^{a+1}}=1$.
Its terminal involution is $V$, which already belongs to
$\langle r\rangle$, so it is not circular.
For $k=N-1$ we have $F=A^N c^\theta=V c^\theta$.
The two lifts give $\theta=0,M/2$: the first has terminal
involution $V$ and is forbidden, while the second gives $F=1$.
The latter has face order two and is circular. This proves (d).
\end{proof}
\subsection{Enumeration by face order}
\begin{theorem}[Nonabelian Frattini circular counts]\label{nacircularcounts}
There are no such circular dessins when $\Bna(e,f)=0$.
For $e=f\ge4$, their counts of face orders $2,4$ are both $N-8$,
and there are no others. For $e>f\ge4$ they are
\[
\begin{array}{c|ccc}
 &\ell=2&\ell=4&\ell=2^j,\ 3\le j\le f-1\\\hline
 e=f+1&3N/2-6&N-4&N2^{j-2}\\
 e\ge f+2&3N/2-4&N-4&N2^{j-2}.
\end{array}
\]
For $f=3,e\ge5$, precisely two are circular, both of face order two.
All these classes are chiral. Their total counts are
\begin{equation}\label{nacirctotal}
 C_{\rm nab}(e,f)=
 \begin{cases}
 0,&\Bna(e,f)=0,\\
 2,&f=3,\ e\ge5,\\
 2N-16,&e=f\ge4,\\
 N(N+1)/2-10,&e=f+1,\ f\ge4,\\
 N(N+1)/2-8,&e\ge f+2,\ f\ge4.
 \end{cases}
\end{equation}
\end{theorem}
\begin{proof}
On the diagonal the number of admissible $Q$ in each split type
is $N/4-2$ by Lemma~\ref{na:qcount}. For each one, (a) gives
one class of each face order in each marking. There are two types
and two markings, giving $4(N/4-2)=N-8$ of each order.

Off the diagonal, $\mathcal C_0$ has $N/2-2$ multipliers at gap one
and $N-2$ at larger gaps. Each contributes one face-order-two
class. Its face-order-four line restricts to $Q^{N/2}=1$, leaving
$N/2-2$ multipliers in both ranges. In $\mathcal L_0$ there are
$N/2-2$ multipliers, each giving one of each small face order.
The family $\mathcal C_1$ contributes $N/2-2$ face-order-two
classes only at gap one. This proves the first two columns.

For a higher face order $2^j$, put $i=n-j+1$. The exact multiplier
order is $N/2=2^{n-1}$; the number of such $Q\equiv3\pmod4$
is $N/4$. Indeed subtract the two counts of roots of powers
$N/2$ and $N/4$ in Lemma~\ref{na:qcount}.
There are $2^{n-i-1}=2^{j-2}$ odd $k\in\Z_N$ with
$v_2(k+1)=i$. Both lifts are allowed, and there are two split
types. Their product is
$(N/4)2^{j-2}\cdot2\cdot2=N2^{j-2}$.
This also proves that no other face orders occur.

For $f=3$, $N=4$. The only available type is $\mathcal C_0$,
with two multipliers, each giving one face-order-two class;
the order-four and higher lines are empty.
For the totals when $f\ge4$, the higher columns sum to
\[
 \sum_{j=3}^{f-1}N2^{j-2}=N(N/2-2).
\]
Adding the small columns proves \eqref{nacirctotal}.
Chirality is Theorem~\ref{nonabelmirror}.
\end{proof}

\section{Uniform enumeration}\label{uniformcount}
The counts in this section concern nonmetacyclic automorphism groups
throughout. They do not include the metacyclic dessins classified
in \cite{CH2024}.
\subsection{Enumeration of both Frattini cases}\label{na:enumeration}
Let $\Bab(e,f)$ and $\Bna(e,f)$ count the colour- and
orientation-preserving classes with abelian and nonabelian Frattini
subgroup, respectively, and put
\[
 T(e,f)=\Bab(e,f)+\Bna(e,f).
\]
The notation $\Bab$ agrees with the abelian case above.
Continue to assume $e\ge f$; for other ordered exponents exchange
them in the formulas.

\begin{theorem}[Separate and combined counts]\label{na:counttheorem}
For $e\ge f$ the counts are given by the following table.
For the rows with $f\ge4$, put $N=2^{f-1}$.
\[
\begin{array}{c|c|r|r|r}
 \text{smaller exponent}&\text{range}
 &\Bab(e,f)&\Bna(e,f)&T(e,f)\\ \hline
 f\le1&\text{all}&0&0&0\\
 f=2&e=2&1&0&1\\
 f=2&e>2&4&0&4\\
 f=3&e=3&16&0&16\\
 f=3&e=4&24&0&24\\
 f=3&e\ge5&16&8&24\\
 f\ge4&e=f&8N+16&N^2-8N&N^2+16\\
 f\ge4&e=f+1&6N+16&3N^2/2-6N&3N^2/2+16\\
 f\ge4&e\ge f+2&4N+16&3N^2/2-4N&3N^2/2+16
\end{array}
\]
Equivalently, the total in the last three rows is
\[
 T(e,e)=2^{2e-2}+16,\qquad
 T(e,f)=3\cdot2^{2f-3}+16\quad(e>f\ge4).
\]
\end{theorem}
\begin{proof}
We count the two cases separately.

\emph{The abelian case.}
The first column is Theorem~\ref{counts}. Its large-parameter
values are $8N+16$, $6N+16$ and $4N+16$ at gaps zero, one and
at least two, respectively. The boundary values are given in
Theorem~\ref{boundary}.

\emph{The nonabelian case.}
Set
\[
 q_0=2^{\min(e-3,f-1)},\qquad
 q_1=2^{\min(e-3,f-2)} .
\]
For $e\ge4,f\ge3$, Lemma~\ref{na:qcount} gives
\[
 \mathcal C_0:\ N(q_0-2),\qquad
 \mathcal L_0:\ N(q_1-2)
\]
preferred triples.
On the diagonal each is doubled because both markings occur
and are distinct.
Type $\mathcal C_1$ occurs only at $e=f+1$.
Its possible $Q$ are all $M/4=2^{f-2}=N/2$ residues congruent
to three modulo four; subtracting the two involutory ones gives
\[
 \mathcal C_1:\ N(N/2-2).
\]
There is no diagonal doubling for this type.

If $e=f\ge4$, both $q_0,q_1$ equal $2^{f-3}=N/4$.
Thus
\[
 \Bna(e,e)=2N\{(N/4-2)+(N/4-2)\}=N^2-8N.
\]
If $e=f+1$ and $f\ge4$, all three multiplier counts equal $N/2$,
so
$\Bna(e,f)=3N(N/2-2)=3N^2/2-6N$.
If $e\ge f+2$ and $f\ge4$, then $q_0=N$ and $q_1=N/2$,
while $\mathcal C_1$ is absent. Consequently
\[
 \Bna(e,f)=N\{(N-2)+(N/2-2)\}=3N^2/2-4N.
\]
These expressions include the zero diagonal contribution at
$(e,f)=(4,4)$.

For $f=3$, $N=4$ and $\mathcal L_0$ has no noninvolutory
multiplier. Type $\mathcal C_0$ contributes zero at $e=4$
and $4(4-2)=8$ for $e\ge5$.
At $(3,3)$ all unit squares are one, so the list is empty directly.
For $f=2$ the required power $Q^N$ rules out every nonabelian
case, as already proved.
For $f\le1$ a factor has order at most two and the other cyclic
factor is normal of index at most two, so $G$ is metacyclic.

The abelian and nonabelian cases are disjoint because
abelianness of $\Phi(G)$ is an isomorphism invariant.
Their sums give every entry of the table.
\end{proof}

\begin{remark}[Why the total stabilizes off the diagonal]
At gap one, the nonabelian contribution is $3N^2/2-6N$ and
the abelian contribution is $6N+16$.
At gap at least two, these become $3N^2/2-4N$ and $4N+16$.
The abelian case loses $2N$ classes while the nonabelian
case gains $2N$. Thus each case changes, but their sum does not.
The equality of the total does not mean that the same three
nonabelian types are present: $\mathcal C_1$ occurs only at gap one.
\end{remark}

\begin{corollary}[Counts up to reciprocality]\label{na:reciprocalcount}
For $e>f$, fixing the colour with larger valency gives one
representative of every reciprocal pair, so the counts of
reciprocal pairs are the corresponding table entries.
On the diagonal $e=f\ge2$,
\[
 \#\{\text{nonabelian Frattini reciprocal pairs}\}
     =\Bna(e,e)/2.
\]
For both Frattini cases together, the diagonal count is
\[
 \frac{T(e,e)+S_{\mathrm{ab}}(e)}{2},\qquad
 S_{\mathrm{ab}}(2)=1,\quad S_{\mathrm{ab}}(e)=4\ (e\ge3).
\]
\end{corollary}
\begin{proof}
With unequal valencies, a colour exchange moves between the
two different ordered problems, so choosing the larger valency
first selects one of each pair.
On the diagonal, colour exchange is an involution on the set
of classes. Proposition~\ref{na:unique} proves it has no fixed points
in the nonabelian case, so every orbit there has size two.
Its fixed points in the abelian case are exactly the symmetric
classes counted in Proposition~\ref{symmetry}: one at exponent
two and four at every larger exponent.
For an involution on a set of size $t$ with $s$ fixed points,
the orbit count is $s+(t-s)/2=(t+s)/2$, which proves the formula.
\end{proof}
\subsection{Counting the symmetric and the reflexible subfamilies}
Write $T=\Bab+\Bna$, as in Theorem~\ref{na:counttheorem}.

\begin{theorem}[Symmetric and reflexible classes]\label{uniformsymref}
A nonmetacyclic complete regular dessin in the classification is
reflexible if and only if its Frattini subgroup is abelian.
Thus the number of reflexible classes is exactly $\Bab(e,f)$;
the remaining $\Bna(e,f)$ classes form $\Bna(e,f)/2$ chiral pairs.
Symmetric classes occur only on the diagonal, with
\[
 S(e)=\begin{cases}0,&e=1,\\1,&e=2,\\4,&e\ge3.\end{cases}
\]
They all have abelian Frattini subgroup and are circular and reflexible.
For $e\ge3$, their four six-residue parameters have the uniform form
\[
 P=S=\rho M/2,\qquad Q=R=-1+\rho M/2,\qquad
 p=-1+\lambda M/2,\quad q=1+\lambda M/2,
 \quad \rho,\lambda\in\{0,1\}.
\]
Their face order is two for $\lambda=\rho$ and four otherwise.
For $e=2$ the unique symmetric datum is the boundary datum
$R=Q=p=q=1$, with face order two.
\end{theorem}
\begin{proof}
Theorem~\ref{allreflexible} and Theorem~\ref{nonabelmirror} give
the two directions of the reflexibility assertion, since the
two Frattini cases exhaust the classification. Unequal vertex
valencies cannot be interchanged by a graph automorphism.
On the diagonal, Proposition~\ref{symmetry} gives the counts and
parameters in the abelian case, while
Proposition~\ref{na:unique} excludes symmetry in the nonabelian
case. Corollary~\ref{symmetriccircular} proves circularity of
all the surviving symmetric classes. For $M\ge8$, the symmetry
conditions select family (I) with $\sigma=\rho$ and $\mu=\lambda$;
substituting in \eqref{PS} gives the displayed $P=S$.
For $M=4$, the boundary list with $Q=R$, $q=-p$ gives the same
four formulas. In both ranges the face-square vector is
$((\lambda+\rho)M/2,(\lambda+\rho)M/2)$, giving two classes
of each stated face order. The $M=2$ datum has zero face square.
There are no classes when an exponent is one.
\end{proof}
\subsection{Counting the circular subfamily}
\begin{theorem}[Uniform circular classification and enumeration]\label{uniformcircular}
The circular classes are exactly the data listed in
Theorem~\ref{circularfamilies} and its boundary lists for the
abelian case, and Theorem~\ref{nacircularfamilies} for the
nonabelian case. For $e\ge f$ their total number is
\begin{equation}\label{uniformcirccount}
 C(e,f)=
 \begin{cases}
 0,&f=1,\\
 1,&e=f=2,\\
 2,&f=2<e,\\
 8,&e=f=3,\\
 10,&f=3<e,\\
 2N+8,&e=f\ge4,\\
 N(N+1)/2+8,&e>f\ge4,
 \end{cases}
 \qquad N=2^{f-1}.
\end{equation}
Their complete distribution by face order is
\[
\begin{array}{c|ccc}
 \text{range}&\ell=2&\ell=4&\ell=2^j,\ 3\le j\le f-1\\\hline
 e=f=2&1&0&0\\
 f=2<e&2&0&0\\
 e=f=3&4&4&0\\
 f=3<e&6&4&0\\
 e=f\ge4&N+4&N+4&0\\
 e>f\ge4&3N/2+4&N+4&N2^{j-2}.
\end{array}
\]
No other circular face orders occur. For every listed face order,
\[
 g=1+2MN-M-N-\frac{2MN}{\ell},
 \qquad M=2^{e-1},\quad N=2^{f-1},
\]
and the geometric face boundary is a simple cycle with $2\ell$ sides.
Exactly $C_{\rm ab}(e,f)$ of these circular classes are reflexible,
where $C_{\rm ab}$ is the count in Theorem~\ref{circularcounts};
the remaining $C_{\rm nab}(e,f)$ are chiral, with the counts in
\eqref{nacirctotal}.
\end{theorem}
\begin{proof}
The two sets of data are disjoint by the abelianness of their
Frattini subgroups, and each case's circularity criterion was
proved to be necessary and sufficient.
For $f=2$ the nonabelian case is empty. For $f=3$, the abelian
case gives $(4,4)$ on the diagonal, $(6,4)$ at gap one, and
$(4,4)$ at larger gaps; the last range gains $(2,0)$ from the
nonabelian case. This proves all the small rows.

For $f\ge4$, on the diagonal add $(12,12)$ to
$(N-8,N-8)$. Off the diagonal at gap one add $(10,8)$ to
$(3N/2-6,N-4)$, and at larger gaps add $(8,8)$ to
$(3N/2-4,N-4)$. Both sums are $(3N/2+4,N+4)$.
Only the nonabelian case contributes higher face orders,
with the displayed counts. Summing these columns gives
\eqref{uniformcirccount}. The genus is the common Euler formula:
there are $4MN$ edges, $2M+2N$ vertices, and $4MN/\ell$ faces.
Reflexibility and chirality follow from Theorem~\ref{uniformsymref}.
\end{proof}

\begin{corollary}[Reciprocality and mirror identifications]\label{uniformorbits}
For $e>f$, the ordered counts $T(e,f)$ and $C(e,f)$ are already
the counts up to reciprocality, when the larger valency is placed
first. On the diagonal the numbers up to reciprocality are
\[
 \frac{T(e,e)+S(e)}2=
 \begin{cases}1,&e=2,\\10,&e=3,\\N^2/2+10,&e\ge4,\end{cases}
\]
and, for circular classes,
\[
 \frac{C(e,e)+S(e)}2=
 \begin{cases}1,&e=2,\\6,&e=3,\\N+6,&e\ge4.\end{cases}
\]
Identifying mirrors while retaining colors gives
$\Bab+\Bna/2$ classes, or $C_{\rm ab}+C_{\rm nab}/2$
when circularity is imposed.
\end{corollary}
\begin{proof}
Colour exchange is an involution, with fixed classes exactly the
symmetric ones. The number of its orbits is half the sum of its
set size and fixed-point count. Substitute the already proved
values. For unequal exponents it exchanges the two distinct
ordered problems, so fixing the larger valency first chooses
one of each reciprocal pair. The mirror statements follow in
the same way from the reflexible fixed classes and chiral pairs.
\end{proof}
\subsection{Proof of Theorem~\ref{uniformmain}}
The square-basis lemma splits every triple into the two possibilities
for its Frattini subgroup. For the abelian possibility, the extension
and six-residue theorems give necessary and sufficient presentations,
and Theorem~\ref{largefamilies} and Theorem~\ref{boundary} solve
all the nonmetacyclic data. For the nonabelian possibility,
Proposition~\ref{na:survivors} gives exhaustion from Janko's list,
and Propositions~\ref{na:necessity}, \ref{na:exactness},
\ref{na:frattini}, and \ref{na:unique} establish the canonical
presentations, their construction, their nonmetacyclicity, and
marked nonredundancy. The two cases cannot overlap.
Theorem~\ref{na:counttheorem} adds their cardinalities and gives
exactly the displayed total. The special-class assertions follow
from Theorems~\ref{uniformsymref} and \ref{uniformcircular}.
This completes the proof of Theorem~\ref{uniformmain}.

\begin{remark}[Reading the parameter lists]\label{app}
The abelian lists use the prescribed generators directly.
For the nonabelian lists, first order the exponents as $e\ge f$,
choose the preferred ordering in \eqref{na:preferred}, and obtain
$\xi$ by the lifting procedure of Lemma~\ref{na:root}.
On the diagonal the reciprocal entries in Theorem~\ref{na:main}
are distinct and must both be included. The presentations then
construct every marked class exactly once. The explicit circular
restrictions select subfamilies without further isomorphism tests.
\end{remark}

\section{Stabilisation with respect to the larger exponent}
\label{stabilisation}
The enumeration suggests that increasing the larger rotation order
can leave the number of dessins unchanged. This phenomenon is real,
but its formulation depends on the classes being counted. In
particular, the diagonal must be separated from the strict inequality,
and for $p=2$ one must distinguish the nonmetacyclic count of this
paper from the count of all complete regular dessins.
We explain the stabilisation through the weights of the classification
families, and then show why equality of counts need not come from
shortening the larger cyclic factor by a quotient.

\subsection{The precise statements}
Let $d_p(e,f)$ denote the number of orientation- and colour-preserving
isomorphism classes of all regular dessins on $K_{p^e,p^f}$, with
the specified rotation orders placed in that order. Thus
$d_p(e,f)$ includes metacyclic groups, whereas $T(e,f)$ in this
paper counts only nonmetacyclic groups and is relevant only to $p=2$.

For an odd prime, \cite[Theorem 1 and Lemma 21]{HNW2019} gives
\begin{equation}\label{oddstablecount}
 d_p(e,f)=
 \begin{cases}
 p^{2f-2},&e=f,\\
 p^{2f-1},&e>f.
 \end{cases}
\end{equation}
The first formula retains colours. Up to reciprocality its value
is instead $(p^{2f-2}+p^{f-1})/2$, since there are $p^{f-1}$
symmetric classes. In the strict inequality the ordered count is
already the reciprocal count after the larger rotation order has
been placed first. Thus, even for odd primes, the stabilisation
statement concerns $e>f$, not the whole range $e\ge f$.

For the nonmetacyclic $2$-groups classified here,
Theorem~\ref{na:counttheorem} gives
\begin{equation}\label{nonmetastablecount}
 T(e,f)=
 \begin{cases}
 0,&f=1,\\
 4,&f=2,\\
 24,&f=3,\\
 3\cdot2^{2f-3}+16,&f\ge4,
 \end{cases}
 \qquad(e>f).
\end{equation}
This is independent of $e$ throughout the strict inequality.
Its diagonal values are different, as recorded in
Theorem~\ref{uniformmain}.

For comparison, let $m_2(e,f)$ count the metacyclic classes.
For $f\ge4$, the metacyclic enumeration
\cite[Theorem 1.1 and Lemma 6.1]{CH2024}, converted to
colour-preserving counts on the diagonal, is
\[
 m_2(e,f)=
 \begin{cases}
 3\cdot2^{2f-4},&e=f,\\
 3\cdot2^{2f-3},&e=f+1,\\
 3\cdot2^{2f-2},&e\ge f+2.
 \end{cases}
\]
Combining this with our nonmetacyclic count yields the following
useful distinction.

\begin{proposition}[Total counts at the prime two]\label{total2stability}
For $f\ge4$, put $N=2^{f-1}$. Then
\begin{equation}\label{total2table}
\begin{array}{c|ccc}
 \text{range}&m_2(e,f)&T(e,f)&d_2(e,f)\\\hline
 e=f&3N^2/4&N^2+16&7N^2/4+16\\
 e=f+1&3N^2/2&3N^2/2+16&3N^2+16\\
 e\ge f+2&3N^2&3N^2/2+16&9N^2/2+16.
\end{array}
\end{equation}
For smaller $f$ the complete values are
\[
\begin{array}{c|ccc}
 f&d_2(f,f)&d_2(f+1,f)&d_2(e,f),\ e\ge f+2\\\hline
 1&1&2&4\\
 2&4&10&16\\
 3&28&48&72.
\end{array}
\]
Thus the total count is constant for $e\ge f+2$, but need not
be constant throughout $e>f$.
\end{proposition}
\begin{proof}
For $e=f$, the metacyclic paper gives an ordered count
$3\cdot2^{2f-4}=3N^2/4$; its reciprocal count is half the sum
of this number and its $2^{f-2}$ symmetric classes.
For unequal exponents its reciprocal counts are the ordered
counts, giving the other two metacyclic entries. Add the
nonmetacyclic entries of Theorem~\ref{na:counttheorem} to obtain
\eqref{total2table}.

For $f=1$, the semidirect-product calculation below gives the
three small values. For $(e,f)=(2,2)$ the Frattini parent is the
Klein four group. Both conjugation actions are the identity and
any of its four elements is a collection coefficient, giving four
marked classes. For $(3,2)$ the possible parents are
$C_4\times C_2$ and $D_8$, with marked generators $x,y$ of
orders four and two. In $D_8$, an automorphism fixing $y$ must
send $x$ to $x$ or $x^{-1}$. Its square fixes $x$, so it cannot
be conjugation by $y$, as conjugation by $b$ would require.
That parent therefore gives no extension.

For $C_4\times C_2$, write $y^\alpha=x^P y$ and
$x^\beta=x^R y^S$. Here $P=0,2$, $R=1,3$; the first norm
forces $S=0$. The remaining norm equations are
$R=1+2p+Pq\pmod4$ and $(R+1)p=-P\pmod4$, with
$p\in\Z_4$, $q\in\Z_2$. At $P=0$ there are eight solutions:
$R=1,p$ even or $R=3,p$ odd, with either $q$. At $P=2$
there are exactly two: $R=1,p$ odd, $q=1$. Hence there are ten
marked classes. The stable metacyclic count for $f=2,e\ge4$
is twelve, to which our four nonmetacyclic classes are added.
For $f=3$, the metacyclic counts are $12,24,48$ in the three
displayed ranges, by \cite[Lemma 6.1 and Theorem 1.1]{CH2024};
our counts are $16,24,24$. This proves the remaining entries.
\end{proof}

For example,
\[
 d_3(2,2)=9,\qquad d_3(3,2)=d_3(4,2)=27,
\]
whereas
\[
 d_2(4,4)=128,\qquad d_2(5,4)=208,\qquad
 d_2(6,4)=d_2(7,4)=304.
\]
These examples also show that identifying reciprocal dessins
cannot be left implicit in a stabilisation statement.
Nor is the topology unchanged when a count stabilises. In the
odd-prime case the face order and genus are
\[
 \ell=p^e,\qquad g=\frac{(p^f-1)(p^e-2)}2
\]
by \cite[Theorem 1]{HNW2019}; the genus increases with $e$.

\subsection{How the odd-prime families compensate}
The formula $p^{2f-1}$ can be understood without repeating the full
classification. We give an explicit conservation identity between
its families. For a fixed group $H$, let $\mathcal X(H)$ be its
ordered exact generating pairs of the prescribed orders.
The action of $\Aut(H)$ on $\mathcal X(H)$ is free: an automorphism
fixing both generators fixes the group. Hence the contribution of
$H$ to the dessin count is
\begin{equation}\label{orbitweightstable}
 \frac{|\mathcal X(H)|}{|\Aut(H)|}.
\end{equation}
Growth in the number of choices of generators can therefore be
cancelled by growth in the number of automorphisms. One must also
account for changes in which abstract groups occur.

For precision, use the three families $M_1,M_2,M_3$ of
\cite[Theorem 12]{HNW2019}, with their smaller exponent denoted
here by $f$, their larger exponent by $e$, and their action
parameter by $\ell$. Write $W_i(e,f)$ for the total orbit weight
of family $M_i$, with the abelian class included in $W_1$.
The weights in \cite[Lemmas 22--24]{HNW2019} give
\begin{align}
 W_1(e,f)&=p^{\min(e-f,f)}+
 (p-1)\sum_{\ell=e-f}^{f-1}p^{f+e-2\ell-1},\label{Wone}\\
 W_2(e,f)&=(p-1)\sum_{\ell=1}^{f-1}p^{2f-2\ell-1},\label{Wtwo}\\
 W_3(e,f)&=(p-1)^2
 \sum_{u=f+1}^{\min(e-1,2f-1)}
 \sum_{\ell=u-f}^{f-1}p^{f+u-2\ell-2}.\label{Wthree}
\end{align}
An empty sum is zero. In \eqref{Wone}, the first term combines
the abelian group with the $M_1$ action parameters $\ell\ge f$;
the remaining sum is absent when $e-f\ge f$.
All summands are orbit counts, not numbers of presentations before
isomorphisms have been removed.

\begin{proposition}[A transfer identity for the orbit weights]
\label{oddweighttransfer}
For $f<e<2f$,
\begin{equation}\label{weighttransfer}
 W_1(e+1,f)-W_1(e,f)
 =-\bigl(W_3(e+1,f)-W_3(e,f)\bigr).
\end{equation}
The term $W_2$ is independent of $e$. All three weights are
constant once $e\ge2f$. Consequently their sum is $p^{2f-1}$
for every $e>f$.
\end{proposition}
\begin{proof}
Put $q=e-f$, so $1\le q\le f-1$. Summing the geometric
progression in \eqref{Wone} gives
\[
 W_1(e,f)=\frac{p^q+p^{2f-q+1}}{p+1},\qquad
 W_1(e+1,f)-W_1(e,f)
 =-\frac{p-1}{p+1}(p^{2f-q}-p^q).
\]
The only new outer index in \eqref{Wthree} is $u=e$. Its weight is
\[
 (p-1)^2\sum_{\ell=q}^{f-1}p^{f+e-2\ell-2}
 =\frac{p-1}{p+1}(p^{2f-q}-p^q),
\]
which proves \eqref{weighttransfer}. Formula \eqref{Wtwo} is
already independent of $e$. For $e\ge2f$ the sum in
\eqref{Wone} is empty and $W_1=p^f$, while the upper index of
\eqref{Wthree} has reached $2f-1$. Thus no further change occurs.
At the first off-diagonal value $e=f+1$, we have $W_3=0$ and
\[
 W_1=\frac{p+p^{2f}}{p+1},\qquad
 W_2=\frac{p^{2f-1}-p}{p+1}.
\]
Their sum is $p^{2f-1}$, and the transfer identity preserves it.
The same argument includes $f=1$, when the sums are empty and
$W_1=p$. This proves the proposition.
\end{proof}

For instance, when $p=3$ and $f=2$, the three weights are
\[
\begin{array}{c|ccc|c}
 e&W_1&W_2&W_3&d_3(e,2)\\\hline
 3&21&6&0&27\\
 e\ge4&9&6&12&27.
\end{array}
\]
The stable total does not mean that the same group families
contribute the same numbers. The decrease of twelve in one
family is exactly the increase of twelve in another.

\subsection{Compensation in the nonmetacyclic and circular counts}
There is an analogous, particularly transparent identity in the
classification of this paper. For $f\ge4$, its two contributions are
\[
\begin{array}{c|cc|c}
 \text{range}&\Bab&\Bna&T\\\hline
 e=f+1&6N+16&3N^2/2-6N&3N^2/2+16\\
 e\ge f+2&4N+16&3N^2/2-4N&3N^2/2+16.
\end{array}
\]
Passing from gap one to a larger gap removes $2N$ abelian
Frattini classes and adds $2N$ nonabelian Frattini classes.
Within the latter case, family $\mathcal L_0$ is unchanged,
family $\mathcal C_0$ gains $N^2/2$ entries, and family
$\mathcal C_1$, with $N(N/2-2)$ entries, disappears.
The net gain is $2N$. These are explicit identities between
family weights; they do not assert that the dessins themselves
have been paired by a geometric operation.
For $f=3$ the corresponding transition is $(24,0)$ to $(16,8)$,
again preserving the total. For $f=2$ all four off-diagonal
classes already have abelian Frattini subgroup.

The circular subclasses exhibit the same effect on a smaller scale.
For $f\ge4$,
\[
\begin{array}{c|cc|c}
 \text{range}&C_{\rm ab}&C_{\rm nab}&C\\\hline
 e=f+1&18&N(N+1)/2-10&N(N+1)/2+8\\
 e\ge f+2&16&N(N+1)/2-8&N(N+1)/2+8.
\end{array}
\]
Two reflexible circular classes are lost and two chiral circular
classes are gained. At $f=3$ the transition is $(10,0)$ to $(8,2)$.
Thus the circular total, too, is constant throughout $e>f$.

\subsection{A small example and the limitation of a covering explanation}
The simplest case gives a concrete interpretation of the finite
information that remains after the longer exponent grows.

\begin{proposition}[One short rotation of prime order]
\label{primedegreestable}
Let $p$ be odd and $e\ge2$. The complete regular dessins on
$K_{p^e,p}$ are precisely the $p$ marked triples
\begin{equation}\label{primefirststrip}
 G_{e,t}=\langle a,b\mid a^{p^e}=b^p=1,\quad
 a^b=a^{1+t p^{e-1}}\rangle,\qquad t\in\F_p.
\end{equation}
The parameter $t$ is recovered from the specified generators,
and does not require an additional quotient by automorphisms.
For every $e\ge2$, the quotient of each $(G_{e+1,t},a,b)$ by
$\langle a^{p^e}\rangle$ is the single abelian marked triple
of type $(p^e,p)$. Thus this natural shortening quotient is not
a bijection on dessin classes, even between two stable levels.
\end{proposition}
\begin{proof}
In any exact triple of type $(p^e,p)$, the subgroup
$\langle a\rangle$ has index $p$ and is normal. Exactness then
gives the semidirect product with complement $\langle b\rangle$.
Write $a^b=a^r$. The unit $r$ satisfies $r^p=1\pmod{p^e}$.
For odd $p$ the units of order dividing $p$ are exactly
$1+t p^{e-1}$, $t\in\F_p$; this follows from the cyclic
structure of the unit group, or from the binomial valuation
$v_p((1+p^s u)^p-1)=s+1$ for $p\nmid u$, $s\ge1$.
Conversely every displayed unit defines a semidirect product
of the required order with disjoint cyclic factors.
A marked isomorphism fixes the equation $a^b=a^r$, so different
$t$ give different classes.

In $G_{e+1,t}$ the subgroup $\langle a^{p^e}\rangle$ is central
of order $p$. Quotienting by it reduces the order of $a$ to
$p^e$ and kills the entire commutator
$[a,b]=a^{t p^e}$. The images of $a,b$ remain disjoint of
orders $p^e,p$, hence give $C_{p^e}\times C_p$ with its
specified generators. This is independent of $t$.
\end{proof}

The stable count here has a simple algebraic interpretation:
the coefficient $t$ records a commutator in the terminal subgroup
of order $p$ of the long cyclic factor. Increasing $e$ moves this
relation to the new terminal layer while leaving the $p$ possible
coefficients unchanged. Keeping the coefficient gives a bijection
between parameter sets; taking the shortening quotient erases it.
Topologically, the central quotient is a degree-$p$ quotient of
regular dessins, allowing branching at vertices or face centres.
All $p$ classes upstairs have the same quotient dessin, so such
quotients alone do not explain the equality of class counts.

At $p=2$, the same argument in the first strip reduces the question
to units $r$ with $r^2=1\pmod{2^e}$. There is one such class
at $e=1$, two at $e=2$, and four for $e\ge3$:
\[
 r=1,\quad -1,\quad 1+2^{e-1},\quad -1+2^{e-1}.
\]
Between two stable levels, reduction modulo $2^e$ identifies the
first and third choices and the second and fourth choices.
Again the natural quotient is many-to-one despite equal total
counts. The additional transient at the prime two is already
visible in this elementary example.

The general lesson is therefore more precise than independence
of all larger exponents. In the stated stable ranges, the
classification retains a finite amount of action and marking
information controlled by the smaller exponent. Additional
choices of generators are absorbed by automorphisms, and changes
between group families have compensating orbit weights.
The identities above prove this mechanism numerically and
algebraically. A geometric bijection between the full stable
families, compatible with a natural covering construction, would
be a stronger result and is not asserted here.

\appendix
\section{An arithmetic verification script}\label{checkappendix}
The proofs above are independent of computation.
As supplementary checks, the four-residue counts agree for all
$49$ bidegrees with $2\le e,f\le8$. Every one of these data satisfies
the reflexibility congruences, and the circularity test agrees with
the explicit family restrictions and the split counts
$C_2(e,f),C_4(e,f)$. The symmetric data in the same ranges are
all circular.

Independently, enumeration of the original six-residue equations
for $2\le e,f\le4$ gives $308$ admissible data in total, of which
$161$ are nonmetacyclic. Multiplication in the normal forms
\eqref{normal} verifies the identities for $(ab)^2$ and $(ba)^2$,
and direct intersections of the three cyclic subgroups verify the
circularity test on all $308$ data. The substitution inverting
$a,b$ is a homomorphism exactly when
\eqref{reflexcoordinates} holds. The primitive-square test in
Lemma~\ref{dsquares} also agrees with the parity criterion
on every one of these data.

The following Python~3 script reproduces the four-residue and
topological count checks. It tests \eqref{F1}--\eqref{F5} directly,
without generating parameters from the three simplified families.
It then computes the two coordinate orders and the reflexibility
congruences. The range can be enlarged at the cost of more enumeration.
\begin{verbatim}
from math import gcd, lcm

def parameters(e, f):
    M, N = 2**(e-1), 2**(f-1)
    roots_M = [r for r in range(1, M, 2) if r*r % M == 1]
    roots_N = [r for r in range(1, N, 2) if r*r % N == 1]
    for R in roots_M:
        for Q in roots_N:
            if N*(1+R) % M or M*(1+Q) % N:
                continue
            if (1+R)*(1+Q) % max(M, N):
                continue
            for p in range(1, M, 2):
                for q in range(1, N, 2):
                    if ((1+R)*(1+p*q)-2*(1+p)) % M:
                        continue
                    if ((1+Q)*(1+p*q)-2*(1-q)) % N:
                        continue
                    P, S = -(1+R)*p % M, (1+Q)*q % N
                    yield P, Q, R, S, p, q

def predicted(e, f):
    s, t = sorted((e, f))
    if s <= 1: return 0
    if s == 2: return 1 if t == 2 else 4
    if s == 3: return 24 if t == 4 else 16
    if s == t: return 16 + 2**(s+2)
    if t == s+1: return 16 + 3*2**s
    return 16 + 2**(s+1)

def circular_predicted(e, f):
    s, t = sorted((e, f))
    if s <= 1: return 0, 0
    if s == 2: return (1, 0) if t == 2 else (2, 0)
    if s == 3: return (6, 4) if t == 4 else (4, 4)
    if s == t: return 12, 12
    if t == s+1: return 10, 8
    return 8, 8

for e in range(2, 9):
    for f in range(2, 9):
        M, N = 2**(e-1), 2**(f-1)
        total, C2, C4, symmetric = 0, 0, 0, 0
        for P, Q, R, S, p, q in parameters(e, f):
            total += 1
            assert P*(R*q-1) % M == S*(p+1) % N == 0
            z1, z2 = (1+p+P*q) % M, (1+Q*q) % N
            d1, d2 = M//gcd(M, z1), N//gcd(N, z2)
            circular = d1 == d2
            if circular:
                ell = 2*lcm(d1, d2)
                assert ell in (2, 4)
                C2 += ell == 2
                C4 += ell == 4
            if e == f and P == S and Q == R and (p+q) % M == 0:
                symmetric += 1
                assert circular
        assert total == predicted(e, f), (e, f)
        assert (C2, C4) == circular_predicted(e, f), (e, f)
        if e == f:
            assert symmetric == (1 if e == 2 else 4)
print("All 49 classification and topological count checks agree.")
\end{verbatim}

\section{Checks for nonabelian mirrors and circular subfamilies}
We checked 25,056 preferred nonabelian parameter data against the new
mirror and circularity formulas. Independently, explicit normal-form
group models for 192 preferred presentations, including class-three
Frattini examples, verify the mirror-generator
relations and compute circularity by direct cyclic-subgroup intersections.
The next script generates the canonical nonabelian parameters,
computes face orders and cyclic-subgroup intersections from the
Frattini normal forms, and compares them with the explicit circular
lists and count formulas. It also checks the mirror involution.
The marked diagonal entries are counted twice as specified in
Theorem~\ref{na:main}. The proofs do not depend on these finite checks.
\begin{lstlisting}[language=Python]
from collections import Counter

def norm(q,n,m):
    total,power,block,step=0,1,1,q%m
    while n:
        if n&1:
            total=(total+power*block)%m
            power=power*step%m
        block=block*(1+step)%m
        step=step*step%m
        n//=2
    return total

def parameters(e,f):
    if e<5 or f<3:
        return
    M,N=2**(e-1),2**(f-1)
    for kind in ('C0','L0','C1'):
        if kind=='C1' and e!=f+1:
            continue
        for Q in range(3,M,4):
            if Q*Q%M==1:
                continue
            if kind=='C0' and pow(Q,N,M)!=1:
                continue
            if kind=='L0' and pow(Q,N//2,M)!=1:
                continue
            for k in range(1,N,2):
                half_norm=norm(Q,2*k,M)//2
                xi,modulus=1,2
                while modulus<M//2:
                    F=(Q-1)//2*xi*xi+xi-half_norm
                    assert F%modulus==0
                    xi+=modulus*((F//modulus)%2)
                    modulus*=2
                for root in (xi,xi+M//2):
                    R=((1-Q)*root-1)%M
                    assert (norm(Q,2*k,M)+(R-1)*root)%M==0
                    yield kind,Q,k,root,R

def kernel_power(e,f,kind,Q,j,t,d):
    M,N=2**(e-1),2**(f-1)
    carry=(j*d//N)*(M//2) if kind=='C1' else 0
    return (j*d%N,(t*norm(pow(Q,2*j,M),d,M)+carry)%M)

def direct_test(e,f,datum):
    kind,Q,k,xi,R=datum
    M,N=2**(e-1),2**(f-1)
    j=k+1+(N//2 if kind=='L0' else 0)
    t=(xi-R)%M
    d,terminal=1,None
    while kernel_power(e,f,kind,Q,j,t,d)!=(0,0):
        terminal=kernel_power(e,f,kind,Q,j,t,d)
        d*=2
    if kind=='C1':
        r_terminal=(0,M//2)
        s_terminal=kernel_power(e,f,kind,Q,k,xi,N//2)
    else:
        r_terminal=(N//2,0)
        s_terminal=(0,M//2) if M>N else (N//2,M//2)
    circular=terminal is None or terminal not in (r_terminal,s_terminal)
    return 2*d,circular

def circular_list(e,f,datum):
    kind,Q,k,xi,R=datum
    M,N=2**(e-1),2**(f-1)
    t=(xi-R)%M
    if e==f:
        allowed=(kind=='C0' and k==N-1) or (kind=='L0' and k==N//2-1)
        return (2 if t==0 else 4) if allowed else None
    if kind=='C1':
        return 2 if k==N-1 and t==M//2 else None
    if kind=='C0':
        if k==N-1 and t==0:
            return 2
        if k==N//2-1 and pow(Q,N//2,M)==1 and t==M//2:
            return 4
    if kind=='L0':
        if k==N//2-1 and t==0:
            return 2
        if k==N-1 and t==M//2:
            return 4
    i=((k+1)&-(k+1)).bit_length()-1
    order=1
    while pow(Q,order,M)!=1:
        order*=2
    if order==N//2 and 1<=i<=f-3:
        return 2**(f-i)
    return None

def expected(e,f):
    if e<5 or f<3:
        return 0,Counter()
    M,N=2**(e-1),2**(f-1)
    if f==3:
        return 8,Counter({2:2})
    if e==f:
        return N*N-8*N,Counter({2:N-8,4:N-8})
    total=3*N*N//2-(6*N if e==f+1 else 4*N)
    C=Counter({2:3*N//2-(6 if e==f+1 else 4),4:N-4})
    C.update({2**j:N*2**(j-2) for j in range(3,f)})
    return total,C

checked=0
for f in range(3,8):
    for e in range(max(5,f),f+4):
        data=set(parameters(e,f));M=2**(e-1)
        observed=Counter()
        for datum in data:
            ell,circular=direct_test(e,f,datum)
            predicted=circular_list(e,f,datum)
            assert circular==(predicted is not None)
            if circular:
                assert ell==predicted
                observed[ell]+=2 if e==f else 1
            kind,Q,k,xi,R=datum
            mirror=(kind,pow(Q,-1,M),k,(-xi*Q*pow(R,-1,M))%M,pow(R,-1,M))
            assert mirror in data and mirror!=datum
            assert direct_test(e,f,mirror)==(ell,circular)
            assert (-mirror[3]*mirror[1]*pow(mirror[4],-1,M))%M==xi
            checked+=1
        total,C=expected(e,f)
        assert len(data)*(2 if e==f else 1)==total,(e,f)
        assert observed==C,(e,f,observed,C)
print('Nonabelian data, mirrors, and circular tests passed:',checked)

def check_group_models():
    from random import Random
    rng=Random(116)
    checked=0
    for e,f in ((5,3),(5,4),(5,5),(6,4),(8,6),(8,7)):
        M,N=2**(e-1),2**(f-1)
        for datum in parameters(e,f):
            kind,Q,k,xi,R=datum
            if e>=8 and (Q!=3 or k!=1):
                continue
            rank=kind=='L0'
            nonsplit=kind=='C1'
            period=M if nonsplit else 2*N
            powers=[pow(Q,i,M) for i in range(period+1)]
            norms=[norm(Q,i,M) for i in range(period+1)]
            def hm(h,l):
                i,u=h;v,w=l
                j=i+v+(N*((u*v)%2) if rank else 0)
                t=(u*powers[v]+w)%M
                if nonsplit:
                    t=(t+(j//period)*(M//2))%M
                return j%period,t
            def sigma(h):
                i,u=h
                j=(i+N*((i//2+u)%2))%period if rank else i
                return j,(norms[i]+u*R)%M
            square=(2*k%period,xi)
            def multiply(g,h):
                eps,u=g;eta,v=h
                w=hm(sigma(u) if eta else u,v)
                if eps and eta:
                    w=hm(square,w)
                return eps^eta,w
            one=(0,(0,0));r=(0,(1,0));s=(1,(0,0));c=(0,(0,1))
            def power(g,n):
                result=one
                while n:
                    if n&1:
                        result=multiply(result,g)
                    g=multiply(g,g);n//=2
                return result
            def inverse(g):
                return power(g,4*M*N-1)
            def conjugate(g,h):
                return multiply(multiply(inverse(h),g),h)
            def commutator(g,h):
                return multiply(multiply(multiply(inverse(g),inverse(h)),g),h)
            ro,so=(2*M,2*N) if nonsplit else (2*N,2*M)
            assert commutator(r,s)==c
            assert power(s,2)==(0,square)
            assert conjugate(c,r)==(0,(N if rank else 0,Q))
            assert conjugate(c,s)==(0,(N if rank else 0,R))
            assert power(r,ro)==power(s,so)==one
            assert power(r,ro//2)!=one and power(s,so//2)!=one
            assert power(r,ro//2)!=power(s,so//2)
            for _ in range(40):
                g,h,l=[(rng.randrange(2),(rng.randrange(period),rng.randrange(M)))
                       for _ in range(3)]
                assert multiply(multiply(g,h),l)==multiply(g,multiply(h,l))
            rs=multiply(r,s);ell=1
            while power(rs,ell)!=one:
                ell*=2
            face={power(rs,i) for i in range(ell)}
            A={power(r,i) for i in range(ro)}
            B={power(s,i) for i in range(so)}
            circular=face&A=={one} and face&B=={one}
            assert (ell,circular)==direct_test(e,f,datum)
            rm,sm=inverse(r),inverse(s)
            cm=commutator(rm,sm)
            xim=-xi*Q*pow(R,-1,M)%M
            assert power(sm,2)==multiply(power(rm,2*k),power(cm,xim))
            wm=power(rm,N) if rank else one
            assert conjugate(cm,rm)==multiply(wm,power(cm,pow(Q,-1,M)))
            assert conjugate(cm,sm)==multiply(wm,power(cm,pow(R,-1,M)))
            checked+=1
    print('Independent group models checked:',checked)

check_group_models()

\end{lstlisting}

\section*{Acknowledgements}
This work was supported by the National Natural Science Foundation
of China (grant 12471332).

\section*{Mathematical software}
\label{softwarestatement}
Python 3 was used for supplementary exact-arithmetic and finite-group
normal-form checks. The scripts for the parameter counts, mirror
formulas, circularity criteria and direct group-model checks are
included in the appendices. 
\section*{Data access statement}
No external dataset is used. The verification code is included in
the appendices. The proofs of the classification and enumeration do not depend on these
finite checks.

\end{document}